\documentclass{amsart}

\DeclareMathOperator{\Lie}{Lie}
\DeclareMathOperator{\val}{val}
\DeclareMathOperator{\Hom}{Hom} 
\DeclareMathOperator{\End}{End}
\DeclareMathOperator{\Int}{Int} 
\DeclareMathOperator{\ad}{ad} 
\DeclareMathOperator{\Ad}{Ad} 
\DeclareMathOperator{\Spec}{Spec} 
\DeclareMathOperator{\Rep}{Rep} 
\DeclareMathOperator{\Norm}{Norm}
\DeclareMathOperator{\Cent}{Cent}
\DeclareMathOperator{\id}{id}

\newcommand{\Fbar}{\overline{F}}

\usepackage{amscd}
\usepackage{amssymb}

\numberwithin{equation}{section}

\newtheorem{theorem}{Theorem}[section]
\newtheorem{corollary}[theorem]{Corollary}
\newtheorem{lemma}[theorem]{Lemma}
\newtheorem{proposition}[theorem]{Proposition}

\theoremstyle{definition}
\newtheorem{definition}[theorem]{Definition} 
\newtheorem{remark}[theorem]{Remark} 
\newtheorem{example}[theorem]{Example}

\begin{document}

\title{Generalized affine Springer fibers}

\author {Robert Kottwitz}\thanks{The research of R.K. was  
supported in part by NSF grant DMS-0245639}
\address{Kottwitz: Department of Mathematics\\ University of Chicago\\ 5734
University Avenue\\ Chicago, Illinois 60637}
\author {Eva Viehmann}
\address{Viehmann: Mathematisches Institut der Universit\"at Bonn \\ Endenicher
Allee 60 \\ 53115 Bonn \\ Germany}

\begin{abstract} This paper studies two new kinds of affine Springer fibers that
are adapted to the root valuation strata of Goresky-Kottwitz-MacPherson.
In addition it develops various linear versions of Katz's Hodge-Newton
decomposition. 
\end{abstract}


\maketitle

\section{Introduction} 
Let $G$ be a connected reductive group over $\mathbb C$. Let $A$ be a maximal
torus in $G$. We write $\mathfrak a$ for the Lie algebra of $A$, $R$ for its set
of roots, and $W$ for its Weyl group. 

For $u \in \mathfrak g=\Lie G$ the Springer fiber over $u$ is the closed
subvariety of the flag manifold of $G$ obtained as the set of Borel subgroups $B$
such that $u \in \Lie B$. For regular $u \in \mathfrak a$ the Springer fiber over
$u$ is the set of Borel subgroups containing $A$, on which $W$ acts simply
transitively.

Now consider the formal power series ring $\mathcal O=\mathbb C [[\epsilon]]$ and
its fraction field $F=\mathbb C((\epsilon))$. To keep this introduction more
readable, we  assume that $G$ is semisimple and simply connected, so that
Iwahori subgroups in $G(F)$ are their own normalizers. For $u \in \mathfrak g(F)$
the affine Springer fiber \cite{kazhdan-lusztig88} over $u$ is the closed subset of
the affine flag manifold obtained as the set of Iwahori subgroups $I$ such that $u
\in \Lie I$. 

The affine Springer fiber over $u$ is nonempty when $u$ is \emph{integral} in the
sense that it is contained in some Iwahori subalgebra (equivalently, in some
parahoric subalgebra). For $u \in \mathfrak a(\mathcal O)$ with regular image in
$\mathfrak a(\mathbb C)$ the affine Springer fiber over $u$ is the set of Iwahori
subgroups containing $A(\mathcal O)$, on which the affine Weyl group $W
\ltimes X_*(A)$ acts simply transitively. 

Now suppose only that $u \in \mathfrak a(F)$ is regular in $\mathfrak g(F)$. As in
\cite{gkm} we attach to $u$ the root valuation function $r_u:R \to \mathbb Z$
defined by $r_u(\alpha)=\val \alpha(u)$. Then $u$ is integral if and only if
every value of $r_u$ is nonnegative (equivalently, if $u \in \mathfrak a(\mathcal
O)$), and in this case the image of $u$ in $\mathfrak a(\mathbb C)$ is regular if
and only if $r_u$ is the constant function $0$. The affine Springer fiber over $u$
becomes more and more complicated as the values of $r_u$ increase.

All the results in this paper were motivated by the desire to find a  
generalized affine Springer theory adapted to a given root valuation
function $r$. More precisely we wanted the generalized affine
Springer fibers over  points in the stratum
\[
\mathfrak a(F)_r:=\{u \in \mathfrak a(F): r_u=r\} 
\]
to  resemble traditional affine Springer fibers over points in the simplest
stratum 
\[
\mathfrak a(F)_0:=\{u \in \mathfrak a(F): r_u=0\}. 
\]

In this paper we give two such generalizations, each of which has it advantages
and disadvantages. By way of motivation we first consider the trivial case in
which $r$ is constant with value $n$. We fix an Iwahori subgroup $I$ containing
$A(\mathcal O)$ and consider the generalized affine Springer fiber 
\begin{equation} \label{eqn.str}
\{g \in G(F)/I : g^{-1}ug \in \epsilon^n\Lie I  \}.
\end{equation}
This is not much of a generalization, since the set \eqref{eqn.str} coincides
with the traditional affine Springer fiber over $\epsilon^{-n}u$, but it does the
job, because now our generalized affine Springer fibers over points in $\mathfrak
a(F)_r$ are the same as the traditional ones over $\mathfrak a(F)_0$.

When $r$ is not constant, the situation becomes much more interesting, and we
consider generalized affine Springer fibers of the form 
 \begin{equation} \label{eqn.strstr}
Z_{r,\lambda}(u):=\{g \in G(F)/K_{r,\lambda} : g^{-1}ug \in \Lambda_{r,\lambda} 
\},
\end{equation} 
where $\Lambda_{r,\lambda}$ is what we call a \emph{root valuation lattice} for
$r$, and where $K_{r,\lambda}$ is the connected normalizer of $\Lambda_{r,\lambda}$
in
$G(F)$. When $r$ is constant with value $n$, the lattice $\epsilon^n\Lie I$ is in
fact a root valuation lattice for $r$, and $K_{r,\lambda}$ turns out to be $I$. 

Root valuation lattices are defined in section \ref{sec.defrvl}, and in section
\ref{sec.conjthm} we prove for them the Conjugation Theorem \ref{thm.conj}. When
translated into the language of generalized affine Springer fibers, this theorem
states that for $u
\in
\mathfrak a(F)_r$ the set
\eqref{eqn.strstr} can be identified with a certain quotient  of $W_r \ltimes
X_*(A)$; here
$W_r$ is the stabilizer of $r$ in $W$, and the quotient 
is  taken modulo the subgroup consisting of elements that can be realized in
the normalizer of $A$ in $K_{r,\lambda}$.  

One disadvantage of root valuation lattices is that  $K_{r,\lambda}$
is usually not a parahoric subgroup, but rather just a subgroup of one, so that
$G(F)/K_{r,\lambda}$ is usually not ind-proper. However they have the advantage
that   $\Lambda_{r,\lambda}$ always contains $\{ k^{-1}uk : k \in
K_{r,\lambda}, u
\in \mathfrak a(F)_r 
\}$ as a Zariski dense open subset (see section \ref{sec.defrvl}), from which it
immediately follows that every element in $\Lambda_{r,\lambda}$ lies in the
closure of the root valuation stratum $\mathfrak g(F)_r=\{ g^{-1}ug : g \in
G(F), u
\in \mathfrak a(F)_r \}.$ 
Thus root valuation lattices give, at least in principle, some information
concerning the difficult unsolved problem of describing closures of root valuation
strata in $\mathfrak g(F)$. From this point of view it is certainly desirable to
have $\Lambda_{r,\lambda}$ be as big as possible. Here the size of a lattice is
measured by its codimension in some standard lattice containing it. 

Our characterization of root valuation lattices in section \ref{sec.defrvl} yields
an upper bound on their size. In section \ref{sec.big} we define explicitly for
each root valuation function  a corresponding root valuation lattice which almost
achieves this upper bound. For groups of small rank we were always able to find an
ad hoc definition of an optimal root valuation lattice whenever our
general construction did not already provide one. This did not lead, however, to a
likewise simple general construction of larger root valuation lattices.

Now we turn to our other generalization of affine Springer theory, again beginning
with the case in which
$r$ is constant with value $n$. Then \eqref{eqn.str} can also be described as 
\[
\{ g \in G(F)/I : g^{-1}ug \in C_n \},
\] 
where 
\[
C_n=\{ v \in \mathfrak g(F) : \ad(v) \Lie I \subset \epsilon^n \Lie I \}.
\] 
Note that $C_n$ is obtained by imposing a condition on the relative position of
the lattice $\Lie I$ and the finitely generated $\mathcal O$-module $\ad(v)\Lie
I$. 

In general we put $N=|R|$ and arrange the $N$ values of $r$ in weakly increasing
order $r_1 \le \dots \le r_N$, and we then define a subset $C_r$ of $\mathfrak
g(F)$ as the set of $v$ such that the relative position of the lattice $\Lie I$ 
and the finitely generated $\mathcal O$-module $\ad(v)\Lie
I$ is
less than or equal to $(r_1,\dots,r_N)$, in the sense that for all $i=1,\dots,N$
the
$i$-th exterior power of $\ad(v)$ maps $\wedge^i(\Lie I)$ into
$\epsilon^{r_1+\dots+r_i}\wedge^i(\Lie I)$. Clearly $C_r$ is Zariski closed and
stable under the adjoint action of $I$; moreover there exists an integer $M \gg 0$
such that
$\epsilon^M\Lie I
\subset C_r \subset \epsilon^{-M}\Lie I$. Using $C_r$, we then obtain another
kind of generalized affine Springer fiber 
\begin{equation}\label{eqn.dagintro}
Y_r(u):=\{ g \in G(F)/I : g^{-1}ug \in C_r \}
\end{equation}
adapted to the given root valuation function $r$. In Theorem \ref{thm.fibers} we
show that the set \eqref{eqn.dagintro} can be identified with the affine Weyl
group when $u \in \mathfrak a(F)_r$. 

One advantage of this construction is that it is canonical. Another is that we  
are now working inside the affine flag manifold, which is ind-proper.
Consequently the set of $u \in \mathfrak g(F)$ for which
\eqref{eqn.dagintro} is non-empty is a closed $G(F)$-invariant subset of
$\mathfrak g(F)$ containing
$\mathfrak g(F)_r$, and hence gives an upper bound for the closure of $\mathfrak
g(F)_r$. Unfortunately Theorem \ref{thm.fibers} also says that for any other root
valuation function $r'$ such that $r'_i=r_i$ for $i=1,\dots,N$, the fibers
\eqref{eqn.dagintro} over $u \in \mathfrak g(F)_{r'}$ are also non-empty, and such
points do not lie in the closure of $\mathfrak g(F)_r$ unless $r'$ lies in the
$W$-orbit of $r$. It is likely that by refining our construction slightly, we
could produce a variant of \eqref{eqn.dagintro} that would be empty for $u \in
\mathfrak g(F)_{r'}$ for $r'$ as above, except when $r'$ is in the $W$-orbit of
$r$. At the moment, however, we see no reason to believe that the resulting
improved upper bound for the closures of root valuation strata  would turn out to
be optimal. Therefore we have chosen to present our construction in its simplest
form. 

We finish our overview of the results in this paper by making two comments.  The
first is that in the body of the paper we also allow partial affine flag
manifolds, for example, the affine Grassmannian. 

The second   concerns  the more general
notion of root valuation strata that was considered in \cite{gkm}. These strata
are indexed by certain pairs $(w,r)$ consisting of an element $w$ in the Weyl group
and a  function $R \to \mathbb Q$. Nontrivial elements $w$ arise when one studies
regular semisimple elements in $\mathfrak g(F)$ whose centralizer in $G$ is a
nonsplit $F$-torus. To keep this paper as simple as possible we have chosen to
ignore root valuation strata with nontrivial $w$ (although we do know that our
results on root valuation lattices can be extended to general $w$). Because we do
not treat nontrivial $w$, we drop them from the notation and index the relevant
strata simply by root valuation functions
$r:R
\to
\mathbb Z$. This is justified by Proposition 4.8.3 in \cite{gkm}, whose content is
that
$w$ is redundant when $r$ takes values in $\mathbb Z$. Recently M.~Sabitova
\cite{sabitova} has proved the much more difficult result that $w$ is always
redundant when $G$ is of classical type.

Now we discuss in greater detail the contents of the various sections of this
paper.  In order to prove Theorem \ref{thm.fibers} we need two kinds of input. The
first is a linear (rather than $\sigma$-linear) version of Katz's Hodge-Newton
decomposition that applies to all endomorphisms, not just invertible ones. This
theory is developed in the first three sections of the paper, and along the way we
prove yet another version of the Hodge-Newton decomposition, a linear one that
makes sense for split connected reductive groups. 

The second kind of input involves methods of recognizing points in the building of
$G(F)$ that come from points in the building of a given Levi subgroup. For special
points such a result was proved in \cite{kottwitz03}, though there is an error in
the proof given there (see Remark \ref{rmk.error}). For arbitrary points in the
building we prove such a result in Corollary \ref{cor.recog}. In Theorem
\ref{thm.comim} we derive from Corollary \ref{cor.recog} another way to recognize
points in the building coming from a given Levi subgroup. It is Theorem
\ref{thm.comim} and our new kind of Hodge-Newton decomposition (Theorem
\ref{VTlambda}) together that yield Theorem \ref{thm.fibers}. 

In section \ref{sec7} we review root valuation functions and strata in greater
detail. In section \ref{sec8} we prove Theorem \ref{thm.fibers}, which we have
already discussed. Sections \ref{sec9} and \ref{sec10} provide material that will
be needed in order to prove the Conjugation Theorem. In section \ref{sec.defrvl}
root valuation lattices are defined and characterized. In section
\ref{sec.conjthm} the Conjugation Theorem that we have already discussed is
proved. The main idea in the proof already appears in the key Lemma
\ref{lem.keyconj}, which is basically the induction step needed in our inductive
proof of the Conjugation Theorem. Finally, in section \ref{sec.big} we prove that
big root valuation lattices do exist; this involves some interesting root system
combinatorics of a kind that was completely unfamiliar to us, related to the
function $r_m$ defined in that section. 

In the first five sections there is no need to limit ourselves to $\mathcal O
=\mathbb C[[\epsilon]]$. Instead we work over  
 a complete discrete valuation ring  $\mathfrak o$. We denote by $F$ the
field of fractions of $\mathfrak o$ and choose an algebraic closure $\Fbar$ of
$F$. We also choose a uniformizing element $\varpi \in \mathfrak o$. The valuation
$\val:F^\times \to \mathbb Z$ (normalized so that the valuation of $\varpi$ is $1$)
extends uniquely to a valuation, still denoted by $\val$, from $\Fbar^\times$ onto
$\mathbb Q$. Finally, we denote by $k$ the residue field $\mathfrak o/\varpi
\mathfrak o$. 

Once we start to consider root valuation strata (section \ref{sec7} and beyond), we
will be working with the particular discrete valuation ring
$\mathcal O:=\mathbb C[[\epsilon]]$ and its fraction field $F=\mathbb C
((\epsilon))$. Needless to say we could replace $\mathbb C$ by any algebraically
closed field of characteristic $0$ and nothing would change. 

We warn the reader (and will repeat the warning later) that while $\mathfrak a$
usually denotes the Lie algebra of the maximal torus $A$, there is one section of
the paper, namely section \ref{sec4}, in which we use $\mathfrak a$ to denote 
$X_*(A) \otimes \mathbb R$ instead. This should cause no confusion since the Lie
algebra of $A$ never arises in that section of the paper. 

It is a pleasure to thank Sandeep Varma for his helpful comments on a preliminary
version of this paper. The second author was partly supported by the SFB/TR 45
``Periods, Moduli Spaces and Arithmetic of Algebraic Varieties'' of the DFG. She
also gratefully acknowledges the hospitality of the University of Chicago during
two visits, and the Hausdorff Center for Mathematics, Bonn for funding the second
one. 

\section{Slopes and Newton homomorphisms}

\subsection{Slopes} 
The roots in $\Fbar$ of any irreducible polynomial with coefficients in $F$ and
nonzero constant term all have the same valuation. Therefore any monic polynomial
$f$ with coefficients in $F$ and nonzero constant term factorizes uniquely as 
\begin{equation}\label{eq.slopefact}
f=\prod_{a \in \mathbb Q} f_a,
\end{equation}
 where $f_a$ is a monic polynomial with coefficients in $F$ all of whose roots in
$\Fbar$ have valuation $a$. We refer to this as the \emph{slope factorization} of
$f$. 

Now consider a finite dimensional $F$-vector space $V$ and an invertible linear
transformation $T:V \to V$. We then have the slope factorization
\eqref{eq.slopefact} of the characteristic polynomial $f$ of $T$. The Chinese
remainder theorem then yields the \emph{slope decomposition} 
\begin{equation}\label{eq.sldcp}
V=\bigoplus_{a \in \mathbb Q} V_a
\end{equation} 
with $V_a:= \ker[f_a(T): V \to V]$. Over $\Fbar$ the subspace $V_a$ is the direct
sum of all the generalized eigenspaces for $T$ on $V$ obtained from eigenvalues
having valuation $a$.

\subsection{Newton homomorphisms}  

Let $H$ be a linear algebraic group over $F$. Let $\mathbb D=\Spec F[\mathbb Q]$
be the diagonalizable group scheme over $F$ whose character group is $\mathbb Q$.  
To any $\gamma \in H(F)$ we are going to associate a Newton homomorphism
$\nu_\gamma:\mathbb D \to H$ (defined over $F$). We use a Tannakian approach, as
in section 4 of \cite{kottwitz85}. 

Write $\Rep H$ for the neutral Tannakian category of finite dimensional
representations of $H$. For any $V$ in $\Rep H$ our element $\gamma \in H(F)$ acts
on $V$ by an automorphism that we will denote by $\gamma_V$. We then have the slope
decomposition \eqref{eq.sldcp} with respect to  $\gamma_V$. 

In fact this construction  lifts the canonical fiber functor on $\Rep H$ to a
$\otimes$-functor from $\Rep H$ to the tensor category of $\mathbb Q$-graded
$F$-vector spaces, or, in other words, we now have a $\otimes$-functor $\Rep H \to
\Rep \mathbb D$ that is strictly compatible with the canonical fiber functors on
those two categories. By Tannakian theory this in turn yields an $F$-homomorphism
$\nu_\gamma: \mathbb D \to H$, the \emph{Newton homomorphism},
characterized as follows. Let $R$ be any $F$-algebra and let $d \in \mathbb D(R)$.
Thus $d$ is a homomorphism $a \mapsto d_a$ from $\mathbb Q$ to $R^\times$. Then
for any $V$ in $\Rep H$ and any $a \in \mathbb Q$ the automorphism
$\nu_\gamma(d)_V$ of
$V_R:=V\otimes_F R$ acts by multiplication by $d_a$ on the direct summand
$(V_a)_R$. 

\begin{lemma}\label{lem.phinu}
Let $\varphi:H \to H'$ be a homomorphism  between linear algebraic groups. Let
$\gamma \in H(F)$ and let $\gamma'\in H'(F)$ denote its image under $\varphi$. 
Then $\nu_{\gamma'}=\varphi \circ \nu_\gamma$. 
\end{lemma}
\begin{proof}
Easy.
\end{proof}
 \begin{lemma}\label{lem.con}
Let $\gamma,\delta \in H(F)$. Let $R$ be any $F$-algebra and suppose that $h \in
H(R)$ satisfies $h\gamma h^{-1}=\delta$. Then $\Int(h) \circ \nu_\gamma
=\nu_\delta$, both sides of this equality being viewed as $R$-homomorphisms from
$\mathbb D$ to $H$. 
\end{lemma} 
\begin{proof}
It is harmless to assume that $R$ is nonzero, in which case $F$ injects into $R$.
Consider $V$ in $\Rep H$. The equality $h\gamma h^{-1}=\delta$ implies that the
characteristic polynomials of $\gamma_V$ and $\delta_V$ are equal. By flatness of
$R$ over $F$ the slope decomposition of $V_R$ with respect to $\gamma_V$ is 
\[
V_R=\bigoplus_{a \in \mathbb Q} \ker[f_a(\gamma_V):V_R \to V_R] 
\] 
and a similar statement holds for $\delta$. The equality $h_V\gamma_V
h_V^{-1}=\delta_V$ then implies that $h_V$ transforms the slope decomposition for
$V_R$ with respect to $\gamma_V$ into the one with respect to $\delta_V$. Since
this is true for all $V$ in $\Rep H$, we conclude that $\Int(h) \circ \nu_\gamma
=\nu_\delta$. 
\end{proof} 

\begin{corollary}\label{cor.gaminM}The centralizer $H_\gamma$ of $\gamma$ in $H$ is
contained in the centralizer  $M$ of $\nu_\gamma$ in $H$. In particular $\gamma$
lies in
$M(F)$. 
\end{corollary}
\begin{proof}
Clear.
\end{proof} 

\begin{definition}
Say that $\gamma \in H(F)$ is \emph{basic}  if its Newton homomorphism
$\nu_\gamma:\mathbb D \to H$ factors through the center of $H$. 
\end{definition} 

\begin{remark}
With notation as in the previous corollary, the Newton homomorphism is central in
$M$, and therefore $\gamma$ is basic in $M(F)$. 
\end{remark} 

\subsection{Newton homomorphisms for tori} Let $T$ be a torus over $F$. We may
restrict the homomorphism 
\begin{equation*}
T(\bar F)=X_*(T)\otimes_\mathbb Z \bar F^\times \xrightarrow{\id\otimes \val}
X_*(T)\otimes_\mathbb Z \mathbb Q
\end{equation*}
to $T(F)$, obtaining a homomorphism 
\begin{equation}\label{eq.homT}
T( F) \rightarrow
X_*(A_T)_\mathbb Q,
\end{equation} 
where $A_T$ is the maximal split subtorus of $T$. 
\begin{proposition}\label{prop.torusnu}
Let $\gamma \in T(F)$. Then $\nu_\gamma \in \Hom(\mathbb D,T)=X_*(A_T)_\mathbb Q$
is the image of $\gamma$ under the homomorphism \eqref{eq.homT}. 
\end{proposition} 
\begin{proof}
When $T$ is split, the proposition follows directly from the Tannakian definition
of Newton homomorphisms. In general one uses the homomorphism from $T$ to its
biggest split quotient in order to reduce to the split case. Of course one must
appeal to Lemma \ref{lem.phinu}. 
\end{proof}

\subsection{Newton points} 
Now assume that $H$ is connected and reductive. Choose a maximal $F$-split torus
$A$ in $H$. The group $\Omega:=\Norm_{H(F)} A/\Cent_{H(F)}A$ then acts on
$X_*(A)_\mathbb Q=\Hom(\mathbb D,A)$, and the natural map $\Hom(\mathbb D,A) \to
\Hom (\mathbb D,H)$ induces a bijection from the set of $\Omega$-orbits in
$X_*(A)_\mathbb Q$ to the set of $H(F)$-conjugacy classes of homomorphisms
$\mathbb D \to H$.  Lemma \ref{lem.con} shows that
 the $H(F)$-conjugacy class of
$\nu_\gamma$ depends only on the $H(F)$-conjugacy class of $\gamma$ and yields a
well-defined orbit of $\Omega$ in $X_*(A)_\mathbb Q$, the \emph{Newton point}
$[\nu_\gamma]$ of (the $H(F)$-conjugacy class of) $\gamma$.

\section{Group-theoretic linear Hodge-Newton decompositions}

 The  purpose of this
section is to prove a linear (as opposed to $\sigma$-linear) version of the
group-theoretic generalization  (see \cite[Remark 4.12]{kottwitz-rapoport02} and
\cite{kottwitz03,viehmann08}) of  Katz's  Hodge-Newton decomposition
\cite{katz79}. In the case of $GL_n$ (the situation considered by Katz, though his
paper is written in the language of linear algebra rather than that of group
theory) we generalize the result in a different direction, treating all
endomorphisms of a given finite dimensional vector space, not just the invertible
ones. This added generality will be needed later, when we apply our linear
Hodge-Newton decomposition to  endomorphisms of the form  
$\ad(u)$.

\subsection{Notation pertaining to split $G$}\label{sub.gp-not}  
For the rest of this section we let $G$ be a split
connected reductive group over $\mathfrak o$ and let $A$ be a split maximal torus
of~$G$ over~$\mathfrak o$.  Fix a Borel subgroup $B=AU$ containing $A$ with
unipotent radical~$U$. As usual, by a standard  parabolic subgroup~$P$ of $G$ we
mean one  containing~$B$, and we write 
$P=MN$, where $M$ is the unique Levi subgroup of~$P$ containing $A$ and $N$ is the
unipotent radical of~$P$. 

Taking scheme theoretic closures, we  obtain natural $\mathfrak o$-structures on
$B,U,P,M,N$ for which $B$ is a Borel subscheme of $G$ over $\mathfrak o$, $M$ is
a split connected reductive group scheme over $\mathfrak o$, and so on. A fact that
will be useful later is that
$P$ is the semidirect product of
$M$ and $N$ over $\mathfrak o$, so that in particular we have $P(\mathfrak
o)=M(\mathfrak o)N(\mathfrak o)$. 

We write $\Lambda_G$ for the 	quotient of $X_*(A)$ by the coroot lattice for $G$,
and we write $p_G$ for the canonical surjection $X_*(A) \twoheadrightarrow
\Lambda_G$.   We 
recall that there is   a natural surjective homomorphism $w_G:G(F)
\twoheadrightarrow 
\Lambda_G$, which can be defined as follows. For
$g\in G(F)$  we define
$r_B(g)
\in X_*(A)$ to be the unique element
$\mu \in X_*(A)$ such that $g \in G(\mathfrak o)  \mu(\varpi)  U(F)$,
and we define
$w_G(g)$ to be the image of $r_B(g)$ under the canonical surjection $p_G$. 

Applying the construction above to $M$ rather than $G$, we obtain $\Lambda_M$, the
quotient of $X_*(A)$ by the coroot lattice for~$M$, and  homomorphisms $p_M:X_*(A)
\twoheadrightarrow \Lambda_M$ and  
$
w_M:M(F)
\twoheadrightarrow  \Lambda_M.
$ 
For
$\mu,\nu \in \Lambda_M$ we write $\mu\overset{P}\le\nu$ if $\nu -\mu$ 
is a non-negative
integral linear combination of 
(images in~$\Lambda_M$ of) coroots $\alpha^\vee$, where
$\alpha$ ranges over the roots of
$A$ in~$N$.

\subsection{Newton points in the split case}
In what follows we will be using Newton points for our
split connected reductive group $G$. Then the Newton point $[\nu_\gamma]$ can be
viewed either as as a Weyl group orbit in $X_*(A)_\mathbb Q$, or as a dominant
element in
$X_*(A)_\mathbb Q$. We will use the two points of view interchangeably.

\begin{example}\label{example}
Suppose that our split group is actually a split torus $A$. Then the Newton
homomorphism $\nu_\gamma \in \Hom(\mathbb D,A)=X_*(A)_\mathbb Q$ is equal to the
image of $w_A(\gamma) \in X_*(A)$ under the canonical inclusion $X_*(A)
\hookrightarrow X_*(A)_\mathbb Q$. 
\end{example}

In general there is still some relationship between $w_G(\gamma)$ and
$\nu_\gamma$, as we will now see. For this we need the $\mathbb Q$-linear map 
\begin{equation}\label{eq.hom11}
X_*(A)_\mathbb Q \twoheadrightarrow \Lambda_G \otimes_\mathbb Z \mathbb Q
\end{equation} 
induced by the canonical surjection $X_*(A)\twoheadrightarrow \Lambda_G$, as well
as the canonical homomorphism 
\begin{equation}\label{eq.hom22}
\Lambda_G \to \Lambda_G \otimes_\mathbb Z \mathbb Q.
\end{equation}
\begin{proposition}\label{prop.imnu} 
Let $\gamma \in G(F)$. Then the image of $[\nu_\gamma]$ under \eqref{eq.hom11} is
equal to the image of $w_G(\gamma)$ under \eqref{eq.hom22}. 
\end{proposition} 
\begin{proof}
Let $C$ be the split torus obtained as the quotient of $G$ by its derived group.
The functoriality of Newton homomorphisms (Lemma \ref{lem.phinu}), together with
the commutativity of the diagram 
\begin{equation*}
\begin{CD}
G(F) @>>> C(F) \\
@V{w_G}VV @V{w_C}VV \\
\Lambda_G @>>> \Lambda_C
\end{CD}
\end{equation*}
reduces us to the case in which our group is a split torus, and this has already
been discussed in Example \ref{example}. 
\end{proof} 
\begin{corollary}\label{cor.bnew} 
Assume that $\gamma \in G(F)$ is basic, so that
its Newton homomorphism
$\nu_\gamma$ lies in $\Hom(\mathbb D,A_G)$, where $A_G$ is the identity component
of the center of $G$. Then under the canonical isomorphism 
\[
\Hom(\mathbb D,A_G)\simeq X_*(A_G)_\mathbb Q \simeq \Lambda_G \otimes \mathbb Q 
\] 
$\nu_\gamma$ corresponds to the image of $w_G(\gamma)$ under \eqref{eq.hom22}.
Thus, in the basic case, $\nu_\gamma$ is uniquely determined by $w_G(\gamma)$. 
\end{corollary}
\begin{proof}
Clear. 
\end{proof} 

\begin{proposition}
Let $\gamma \in G(F)$ and let $\gamma_s \in G(\bar F)$ denote semisimple part of
the Jordan decomposition of $\gamma$. Choose $a \in A(\bar F)$ such that $a$ is
conjugate to $\gamma_s$ in $G(\bar F)$. The Weyl group orbit of $a$ is then
uniquely determined by $\gamma$, and the Newton point of $\gamma$ is the Weyl group
orbit of the image of $a$ under 
\begin{equation*}
A(\bar F)=X_*(A)\otimes_\mathbb Z \bar F^\times \xrightarrow{\id\otimes \val}
X_*(A)\otimes_\mathbb Z \mathbb Q.
\end{equation*}
\end{proposition} 
\begin{proof}
We first reduce to the case in which $\gamma$ is semisimple. In characteristic
$0$, the Jordan decomposition of $\gamma$ is defined over $F$, and it is evident
from the definition of $\nu_\gamma$ that $\nu_{\gamma_s}=\nu_\gamma$. In
characteristic $p$ we use instead the fact that $\gamma^m$ is semisimple when $m$
is a sufficiently big power of the prime $p$, it being again evident that 
$\nu_{\gamma^m}=m\nu_\gamma$. 

Now assume that $\gamma$ is semisimple. Choose a maximal $F$-torus $T$ of $G$
containing $\gamma$. Then the proposition follows from Proposition
\ref{prop.torusnu}. 
\end{proof}

\subsection{Linear Hodge-Newton decomposition}
For any coweight $\mu \in X_*(A)$ (usually taken to be dominant) and any $\gamma
\in G(F)$ we consider the  set 
\begin{equation}\label{aDL}
 X^G_{\mu}(\gamma):=\{ x \in G(F)/G(\mathfrak o) : x^{-1}\gamma x \in G(\mathfrak
o)\mu(\varpi)G(\mathfrak o) \}.
\end{equation}

\begin{theorem} \label{linHNthm}
 Let $\mu \in X_*(A)$ be a dominant coweight, let $P=MN$ be a standard parabolic
subgroup, and let
$\gamma
\in M(F)$. Then the following two conclusions hold: 
\begin{enumerate}
\item  If $X^G_{\mu}(\gamma)$ is non-empty, then $w_M(\gamma)
\overset{P} \le p_M(\mu)$. 
\item Suppose that $w_M(\gamma)=p_M(\mu)$ and that every slope of $\Ad(\gamma)$ on
$\Lie N$ is strictly positive.  Then the natural injection
$X^M_{\mu}(\gamma)
\hookrightarrow X^G_{\mu}(\gamma)$ is a bijection. 
\end{enumerate}  
\end{theorem}
 
  The theorem will be proved at the end of this section. The second part  
of the theorem is our linear  group-theoretic generalization of
Katz's Hodge-Newton decomposition (see Theorem 1.6.1 of
\cite{katz79}). From the first part of the theorem we obtain the following
corollary, which is  analogous to the  group-theoretic generalization of Mazur's
inequality (see Theorem 1.4.1 of
\cite{katz79}) proved by Rapoport-Richartz \cite{rapoport-richartz96}. 
 
\begin{corollary}\label{cor.maz}
Let $\mu$ be a dominant coweight, let $\gamma \in G(F)$, and write $[\nu_\gamma]$
for the  $B$-dominant element in $X_*(A)_\mathbb Q$ obtained as the Newton point of
$\gamma$. Suppose further that
$X^G_{\mu}(\gamma)$ is non-empty. Then $[\nu_\gamma] \le \mu$ in the sense that
$\mu -[\nu_\gamma]$ is a non-negative real linear combination of simple coroots.
Moreover
$w_G(\gamma)=p_G(\mu)$.
\end{corollary}
\begin{proof}
 The Newton homomorphism
$\nu_\gamma$ of $\gamma$ determines a  parabolic subgroup $P=MN$, characterized by
the following two properties: $M$ is the centralizer of $\nu_\gamma$ in~$G$, and 
$\mathbb D$ acts (through 
$\nu_\gamma$) on $\Lie N$ with strictly positive weights.  It is harmless to
replace
$\gamma$ by a $G(F)$-conjugate, so we may assume that $P$ is
standard. 

As we have already seen in Corollary \ref{cor.gaminM}, $\gamma$ lies in $M(F)$ and
is basic for $M$. By Corollary \ref{cor.bnew}, under the canonical isomorphism 
\[
\Hom(\mathbb D,A_M) \simeq X_*(A_M)_\mathbb Q \simeq \Lambda_M\otimes \mathbb Q
\] 
$\nu_\gamma$ corresponds to the image of $w_M(\gamma)$ under $\Lambda_M \to
\Lambda_M\otimes \mathbb Q$, and from the first part of Theorem \ref{linHNthm} 
we conclude that $[\nu_\gamma] \le \mu$. Here we used standard facts (see,
\emph{e.g.}, Lemma 4.9 in  \cite{kottwitz03}) relating the partial order $\le$ on
$X_*(A)_\mathbb Q$ to the one on $\Lambda_M$. 

The last statement of the corollary is obtained by choosing $g \in G(F)$ such that
$g^{-1}\gamma g \in G(\mathfrak o)\mu(\varpi)G(\mathfrak o)$ and then observing
that 
\[
w_G(\gamma)=w_G(g^{-1}\gamma g) =w_G(\mu(\varpi))=p_G(\mu). 
\]
\end{proof}

\subsection{Two lemmas}
We now give two lemmas that will be needed in the proof of the theorem we just
stated. These lemmas are similar to Lemmas 3.1 and 3.2 in \cite{kottwitz03}. 

\begin{lemma}\label{lem.Tspace}
Let $V$ be a finite dimensional $F$-vector space, and let
$T:V
\to V$ be an invertible linear transformation all of whose slopes are strictly
positive. 
 Suppose that $\Lambda$
is a lattice in $V$ such that $T\Lambda \subset \Lambda$. Then 
$1-T:V \to V$ is bijective and moreover $(1-T)\Lambda =\Lambda$. In particular,
if $v \in V$ satisfies $v-Tv \in \Lambda$, then $v \in \Lambda$. 
\end{lemma}
\begin{proof}
Clearly $(1-T)\Lambda \subset \Lambda$. Let
$f(\lambda)=\lambda^d+a_1\lambda^{d-1}+\dots+a_d$ be the characteristic polynomial
of
$T$,
$d$ being the dimension of $V$. Our hypothesis on the slopes guarantees that
$a_1,\dots,a_d$ all lie in the maximal ideal 
$\varpi\mathfrak o$. Therefore $\det(1-T)=f(1)$ lies in $1+\varpi\mathfrak o
\subset\mathfrak o^\times$. It is then clear that
$1-T$ is bijective from $V$ to $V$ and from $\Lambda$ to $\Lambda$. 
\end{proof}

The second lemma is a non-abelian analog of the first. It makes use  of the
notation in subsection \ref{sub.gp-not}, in particular the Borel
subgroup
$B$ containing $A$. 

\begin{lemma} \label{lemma.m.3.1}
Let $\mu \in X_*(A)$ be a coweight that
 is dominant with respect to our chosen $B$, and let $P=MN$  be a standard
parabolic subgroup. We consider elements 
$\gamma \in M(F)$ and
$n
\in N(F)$ satisfying $n^{-1}\gamma n \gamma^{-1} \in N(\mathfrak o)$.  We further
assume that $\gamma \in M(\mathfrak o)\mu(\varpi)M(\mathfrak o)$ and that every
slope of $\Ad(\gamma)$ on $\Lie N$ is strictly positive. The conclusion is then
that $n \in N(\mathfrak o)$.
\end{lemma} 
\begin{proof}
We claim that $\gamma N(\mathfrak o) \gamma^{-1} \subset N(\mathfrak o)$. Indeed,
this follows from our hypothesis that $\gamma \in M(\mathfrak
o)\mu(\varpi)M(\mathfrak o)$. Here we used that 
$M(\mathfrak o)$ normalizes $N(\mathfrak o)$ and that $\mu(\varpi)N(\mathfrak
o)\mu(\varpi)^{-1} \subset N(\mathfrak o)$, a consequence of the fact that $\mu$
is dominant. 

Choose an $M$-stable filtration 
\begin{equation}
 N=N_0 \supset N_1 \supset N_2 \supset \dots \supset N_r=\{1\}
\end{equation}
 by normal subgroups with $N_i/N_{i+1}$ abelian for all $i$. Each
$N_i$ is
$A$-stable, hence is a product of root subgroups (even over $\mathfrak o$). We will
prove by induction on $i$ ($0 \le i \le r$) that $n \in N_i(F) N(\mathfrak
o)$. For
$i=0$ this statement is trivial, and for $i=r$ it is the statement of the lemma.
It remains to do the induction step. So suppose that for $0 \le i  < r$ we can
write $n$ as $n=n_i n_{\mathfrak o}$ for
$n_i \in N_i(F)$ and $n_\mathfrak o \in N(\mathfrak o)$. Then
$n_i^{-1}\gamma n_i\gamma^{-1}
\in N_i(\mathfrak o)$. By Lemma \ref{lem.Tspace}, applied to the $F$-vector
space $(N_i/N_{i+1})(F) \cong \Lie N_i/\Lie N_{i+1}$ and the linear transformation
given by conjugation by
$\gamma$, the image of
$n_i$ in
$(N_i/N_{i+1})(F)$ lies in $(N_i/N_{i+1})(\mathfrak
o)$. Since
$N_i(\mathfrak o)$ maps onto $(N_i/N_{i+1})(\mathfrak o)$, we see that $n_i$ can
be written as $n_{i+1} n'_{\mathfrak o}$ with $n_{i+1} \in N_{i+1}(F)$ and
$n'_{\mathfrak o}
\in N_i(\mathfrak o)$. Thus $n=n_{i+1} \cdot (n'_\mathfrak o  n_\mathfrak o) \in
N_{i+1}(F)N(\mathfrak o)$, as desired.
\end{proof}

\subsection{Proof of Theorem \ref{linHNthm} }
Now we are ready to prove Theorem \ref{linHNthm}. The proof is exactly parallel to
that of Theorem 1.1 in \cite{kottwitz03}. 
 Let $g \in G(F)$ and suppose
that 
\begin{equation}\label{m.3.1} g^{-1}\gamma g \in 
G(\mathfrak o)\mu(\varpi)G(\mathfrak o). 
\end{equation}   Use the Iwasawa decomposition to write $g$ as $mnk$
for $m \in M(F)$, $n \in N(F)$ and $k \in G(\mathfrak o)$. It follows 
from \eqref{m.3.1} that
\begin{equation}\label{m.3.2} n_1m_1 \in G(\mathfrak o)\mu(\varpi)G(\mathfrak o),
\end{equation} where $m_1:=m^{-1}\gamma m \in M(F)$ and
$n_1:=n^{-1}m_1 n m_1^{-1} \in N(F)$. We claim that 
\begin{equation}\label{353}
w_M(\gamma)=p_M(r_B(n_1m_1)).
\end{equation}
 Indeed,
\begin{equation*} w_M(\gamma)=w_M(m_1)=p_M(r_{B\cap M}(m_1))=p_M(r_B(n_1m_1)).
\end{equation*} 

Using equation (2.6) in section 2.6 of \cite{kottwitz03},
together with \eqref{m.3.2} and 
\eqref{353} above, we conclude that
$w_M(\gamma) \overset{P} \le p_M(\mu)$, which proves the first part of the
theorem. 

Now we prove  the second part of the theorem. Under the
hypotheses that $w_M(\gamma)=p_M(\mu)$ and that all the slopes of
$\Ad(\gamma)$ on $\Lie N(F)$ are strictly positive (and with 
$g,m,n,m_1,n_1$ as above), we begin by proving that $g \in M(F) 
G(\mathfrak o)$. It
follows from \eqref{m.3.2}, \eqref{353},  our assumption
$w_M(\gamma)=p_M(\mu)$,  and Lemma 2.2 of \cite{kottwitz03} that
$n_1m_1 \in G(\mathfrak o)  M(F)$. Therefore 
$n_1 \in G(\mathfrak o) M(F)$, say $n_1=k_2m_2$
with $k_2 \in G(\mathfrak o)$ and $m_2 \in M(F)$. Then 
$n_1m_2^{-1} \in P(\mathfrak o)$,
and therefore $n_1 \in N(\mathfrak o)$ and $m_2 \in M(\mathfrak o)$. Since
$n_1 \in N(\mathfrak o)$, the second statement of Lemma 2.2 of \cite{kottwitz03}
applies to $n_1m_1$, and hence $m_1 \in M(\mathfrak
o)\mu(\varpi)M(\mathfrak o)$. 

Now applying Lemma \ref{lemma.m.3.1} (not to the element $\gamma$, but to its
conjugate $m_1$, which satisfies all the hypotheses of that lemma), 
we see that
$n \in N(\mathfrak o)$. Therefore $g= m \cdot nk \in M(F)G(\mathfrak o)$, 
and we are done,
since we have already seen that 
\begin{equation} m^{-1}\gamma m=m_1
\in M(\mathfrak o)\mu(\varpi)M(\mathfrak o),
\end{equation} 
which shows that $g \in X^G_\mu(\gamma)$ is the image of $m \in X^M_\mu(\gamma)$. 
\qed

\section{Linear Hodge-Newton decomposition for endomorphisms}\label{sec3}
The results of the previous section were proved for split groups $G$. For $G=GL_n$
they yield concrete assertions in linear algebra, in particular a linear
Hodge-Newton decomposition for certain triples $(V,T,\Lambda)$ consisting of a
finite dimensional $F$-vector space $V$, a linear bijection $T:V \to V$, and a
lattice
$\Lambda$ in $V$. This will be the content of Theorem \ref{VTlambda}, a  close
relative of Katz's $\sigma$-linear Hodge-Newton decomposition \cite{katz79}.
However, in Theorem \ref{VTlambda} we actually work in greater generality, in that
we drop the requirement that $T$ be invertible, allowing $T$ to be an arbitrary
endomorphism of $V$. This will require us to use slopes in 
$\tilde {\mathbb Q}:=\mathbb Q \cup \{\infty\}$, not just $\mathbb Q$, and so we
must begin with some preliminary remarks about valuations and slopes. 

\subsection{Extended valuation map $\val:\Fbar \to \tilde {\mathbb Q}$} 
We extend $\val:\Fbar^\times \to \mathbb Q$ to $\val:\Fbar \to \tilde {\mathbb Q}$
by putting $\val(0)=\infty$. This extended valuation map is a homomorphism of
monoids, using field multiplication  on $\Fbar$ and   
the usual addition on $\mathbb Q$ together with the rules 
$x+\infty=\infty+x=\infty$ for all
$x$ in $\tilde {\mathbb Q}$. We will also use the convention that
$\varpi^\infty=0$. 

The usual partial order on $\mathbb Q$ extends to a total order on $\tilde
{\mathbb Q}$ for which $\infty$ is the greatest element. 

\subsection{Slopes and Newton points} 
Let $T$ be an endomorphism of an $n$-dimensional $F$-vector space. Then $T$ has
$n$ eigenvalues $\lambda_1,\dots,\lambda_n$ in $\Fbar$. The \emph{slopes} of $T$
are the $n$ elements $\nu_i=\val(\lambda_i)$ of $\tilde {\mathbb Q}$.
Renumbering the slopes so that $\nu_1 \le \nu_2 \le \dots \le  \nu_n$, we
obtain the \emph{Newton point} 
\[
\nu(T):=(\nu_1,\dots,\nu_n)
\]
 of $T$. The Newton point lies in $\tilde {\mathbb Q}^n_+$, the subset of
$\tilde {\mathbb Q}^n$ consisting of non-decreasing $n$-tuples. 

\subsection{Hodge points} 
Now let $T$ be an endomorphism of an $n$-dimensional vector space $V$, and let
$\Lambda$ be a lattice in $V$ (an $\mathfrak o$-submodule of $V$ that is free of
rank $n$). Since $T$ need not be invertible, $T\Lambda$ need not be a lattice in
$V$, but it certainly is a finitely generated $\mathfrak o$-submodule of $V$, and
therefore there exists an $\mathfrak o$-basis $e_1,\dots,e_n$ for $\Lambda$ and
$n$ elements $\mu_1 \le \mu_2 \le \dots \le \mu_n$ of $\mathbb Z \cup \{ \infty
\}$ such that 
\begin{equation}\label{eq.Tmu}
T\Lambda=\mathfrak o\varpi^{\mu_1} e_1 + \dots + \mathfrak o\varpi^{\mu_n} e_n.
\end{equation} 
We are of course making use of our convention  that $\varpi^\infty=0$. In this way
we obtain the \emph{Hodge point} 
\[
\mu(T,\Lambda):=(\mu_1,\dots,\mu_n)
\]
of $(T,\Lambda)$ (independent of the choice of basis for which \eqref{eq.Tmu}
holds). When necessary for the sake of clarity, we will denote the $i$-th entry of
the $n$-tuple $\mu(T,\Lambda)$ by $\mu_i(T,\Lambda)$ rather than $\mu_i$, and the
same goes for $\nu(T)$.  It is worth noting that the number of entries of $\mu$
that are finite is equal to the rank of $T$. 

For any $i \in \{1,2,\dots,n\}$ and any $a \in \mathbb Z \cup \{ \infty \}$, the
inequality $\mu_1+\dots+\mu_i \ge a$ is equivalent to the condition that the
linear transformation $\wedge^iT:\wedge^iV \to \wedge^iV$ map the lattice 
$\wedge^i\Lambda$ into $\varpi^a \wedge^i\Lambda$. An even more concrete version
of this is obtained by choosing an  $\mathfrak o$-basis for $\Lambda$, which also
serves as $F$-basis for $V$ and allows us to regard $T$ as an $n\times n$-matrix.
The matrix entries of $\wedge^iT$ are then the $i\times i$-minors of $T$, and the 
inequality $\mu_1+\dots+\mu_i \ge a$ is also equivalent to the condition that
every $i\times i$-minor of $T$ lies in 
$\varpi^a\mathfrak o$. 

\subsection{Openness lemma for Hodge points} 
Recall that $F$ can be regarded as a topological field. In this topology the sets 
$\varpi^a\mathfrak o$ $(a \in \mathbb Z)$ form a neighborhood base at the point
$0$, and each such subset is both open and closed in $F$. For any affine scheme
$X=\Spec(A)$ of finite type over $F$ ($A$ being a finitely generated
$F$-algebra) the set $X(F)$ then acquires a natural topology: the smallest one
(fewest open sets) for which the functions $f:X(F) \to F$ obtained from elements
in the $F$-algebra $A$ are all continuous. Any morphism $X \to X'$ then induces a
continuous map $X(F) \to X'(F)$. In particular any finite dimensional
$F$-vector space acquires a natural topology; for example $F^n=\mathbb A^n(F)$
acquires the usual product topology. 

\begin{lemma}\label{lem.open}
Let $\Lambda$ be a lattice in an $n$-dimensional $F$-vector space $V$. Let $i \in
\{1,\dots,n\}$ and let $a \in \mathbb Z$. Then the  subsets 
\[\{T \in \End_FV: \mu_1(T,\Lambda)+\dots+\mu_i(T,\Lambda) \ge a\}\]
\[\{T \in \End_FV: \mu_1(T,\Lambda)+\dots+\mu_i(T,\Lambda) = a\}\]
are  open and closed in $\End_FV$. 
\end{lemma}
\begin{proof}
That the first set is open and closed follows from the description of
$\mu(T,\Lambda)$ given at the end of the last subsection. Since the second set is
the difference of two sets of the first kind, it too is open and closed. 
\end{proof}

\subsection{Dominance order on $\tilde {\mathbb Q}^n_+$} \label{sub.domord}
 For $\mu, \nu \in \tilde {\mathbb Q}_+^n$
 we say that $\mu \ge \nu$ if 

\begin{equation}\label{eq.root.ineq}
\mu_1+\dots + \mu_i \le \nu_1+\dots + \nu_i\text{ \quad for
$i=1,\dots,n-1$}
\end{equation}
and 
\begin{equation}\label{eq.root.eq}
\mu_1+\dots + \mu_i = \nu_1+\dots + \nu_i \text{ \quad for $i=n$}.
\end{equation}
The direction of the inequality in \eqref{eq.root.ineq} reflects our convention of
considering \emph{increasing} sequences $\mu_1 \le \mu_2 \le
\dots \le \mu_n$ to be dominant, or, in other words, of taking 
the standard Borel subgroup in $GL_n$ be \emph{lower} triangular. 
 
\subsection{Hodge-Newton decomposition for $(V,T,\Lambda)$} 
In the next theorem the first part is an analog of Mazur's inequality and the
second part is an analog of Katz's Hodge-Newton decomposition. 

\begin{theorem}\label{VTlambda}
Let $V$ be an $n$-dimensional $F$-vector space, let $T:V \to V$ be a linear
transformation, and let $\Lambda$ be a lattice in $V$. Put $\mu=\mu(T,\Lambda)$
and $\nu=\nu(T)$. Then the following two conclusions hold. 
\begin{enumerate}
\item $\nu(T) \le \mu(T,\Lambda)$. 
\item Suppose that $V$ is the direct sum of linear subspaces $U,W$ having the
property that $TU \subset U$ and $TW \subset W$. Suppose further that every slope
of $T$ on $U$ is strictly less than every slope of $T$ on $W$. Finally, put
$r=\dim U$ and suppose that $\mu_1+\dots+\mu_r = \nu_1+\dots+\nu_r$. Then 
$\Lambda$ decomposes as 
\[
\Lambda = (\Lambda \cap U) \oplus (\Lambda \cap W).
\]
\end{enumerate}
\end{theorem} 
\begin{proof}
We will first prove the theorem in the special case that $T$ is invertible,
deriving it from the case $G=GL_n$ of our group theoretic linear Hodge-Newton
decomposition. We will then derive the general case from this special one, 
perturbing $T$ by a suitably small non-zero scalar in such a way that the
perturbed linear transformation is invertible and that (1),(2) for it imply (1),(2)
for $T$. 

So for the moment we assume that $T$ is invertible. Choose any $F$-basis for $V$,
and use it to identify $V$ with $F^n$ and the algebraic group $GL_FV$ with
$G=GL_n$. Choose $g \in G(F)/G(\mathfrak o)$ such that $\Lambda =g \mathfrak o^n$.
Our linear transformation $T$ is now an element $\gamma \in G(F)$, and $\mu$
(which lies in $\mathbb Z^n$ and satisfies $\mu_1 \le \dots \le \mu_n$) can be
viewed as a coweight of the diagonal maximal torus $A$ in $G$ that is dominant with
respect to the \emph{lower} triangular Borel subgroup $B$ of $G$; moreover we
have $g^{-1}\gamma g \in G(\mathfrak o)\mu(\varpi)G(\mathfrak o)$. Corollary
\ref{cor.maz} tells us that $[\nu_\gamma] \le \mu$ in the dominance order for $B$,
and this is the content of (1). 

For (2) we choose our basis $e_1,\dots,e_n$ for $V$ in such a way that the first
$r$ elements form a basis for $U$ and the remaining $n-r$ elements form a basis
for $W$. Still taking $B$ to be the lower triangular Borel subgroup, we obtain a
standard parabolic subgroup $P=MN$ with $P$ consisting of elements $g$ in $G$ such
that $gW=W$, and $M$ consisting of elements $g$ such that $gU=U$ and $gW=W$. 
Then $\gamma=T \in M(F)$, and our hypothesis that the slopes of $\gamma$ on $U$
are strictly smaller than those on $W$ is equivalent to the statement that every
slope of $\Ad(\gamma)$ on $\Lie N$ is strictly positive. The hypothesis that
$\mu_1+\dots +\mu_r=\nu_1+\dots + \nu_r$ is equivalent to the hypothesis that
$w_M(\gamma)=p_M(\mu)$ (bear in mind that the equality $\mu_1+\dots +\mu_n=\nu_1
+\dots +\nu_n$ is automatic). As before we choose $g \in G(F)$ such that
$\Lambda=g\mathfrak o^n$. Then $gG(\mathfrak o)$ lies in $X^G_\mu(\gamma)$, and
thus Theorem \ref{linHNthm}(2) tells us that $g \in M(F)G(\mathfrak o)$, which
implies that 
\[
\Lambda = (\Lambda \cap U) \oplus (\Lambda \cap W),
\] 
as desired. 

So now we know that the theorem is true when $T$ is invertible. For a general
endomorphism $T$ our method will be to perturb $T$ slightly, replacing it by
$\tilde T=T+c$ for suitably small $c \in F$ with $c \ne 0$. We will do this in
such a way that $\tilde T$ is invertible and the statements of the theorem for
$\tilde T$ imply those for $T$. Put $\tilde \mu=\mu(\tilde T,\Lambda)$ and $\tilde
\nu=\nu(\tilde T)$; these of course depend on $c$. 

We begin with (1). We must show that for all $i \in \{1,\dots,n\}$ 
\begin{equation}\label{eq.lei}
\mu_1+\dots+\mu_i \le \nu_1 +\dots +\nu_i.
\end{equation} 
and that 
\[
\mu_1+\dots +\mu_n = \nu_1 + \dots + \nu_n. 
\]
Let us start by verifying the last equality. We have already handled the case in
which $T$ is invertible, and if $T$ is not invertible then it is clear from the
definitions that both sides of the equality are infinite. 

Next we prove the inequality \eqref{eq.lei}. 
If its right side  is infinite, the inequality is trivially true, so
we may as well assume that it is finite. Then the rank of $T$ is at
least $i$, so that $\mu_1+ \dots +\mu_i$ is finite as well. By Lemma
\ref{lem.open}, for all sufficiently small $c \in F$ we have 
\[
\mu_1+\dots +\mu_i=\tilde \mu_1+\dots +\tilde \mu_i. 
\] 
The eigenvalues of $\tilde T$ are obtained from those of $T$ by adding $c$.
Therefore, for all sufficiently small $c \in F$, $c\ne 0$, the perturbed linear
transformation $\tilde T$ is invertible and 
\[
\nu_1+\dots +\nu_i=\tilde \nu_1+\dots +\tilde \nu_i.
\] 
Since $\tilde T$ is invertible, (1) holds for it, which means that 
\[
\tilde\mu_1+\dots+\tilde\mu_i \le \tilde\nu_1 +\dots +\tilde\nu_i.
\]
 and therefore
the inequality
\eqref{eq.lei} follows. 

It remains to prove (2). The case $W=0$ being trivial, we may as well assume that
$r < n$. Since every slope of $T$ on $U$ is strictly less than every slope of $T$
on $W$, the slopes of $T$ on $U$ (resp. $W$) are $\nu_1,\dots,\nu_r$ (resp.
$\nu_{r+1},\dots,\nu_n$), and we have 
\[
\nu_1 \le \dots \le \nu_r < \nu_{r+1} \le \dots \le \nu_n,
\]
from which we conclude that all of $\nu_1, \dots, \nu_r$ are finite. This implies
that the rank of $T$ is at least $r$ and hence that $\mu_1,\dots,\mu_r$ are also
finite. From Lemma \ref{lem.open} it follows that when $c$ is sufficiently small,
there is an equality 
\[
\mu_1+\dots + \mu_r=\tilde \mu_1+ \dots + \tilde \mu_r. 
\]

For any $c \in \mathfrak o$ we have $\tilde TU \subset U$ and $\tilde TW \subset
W$. For sufficiently small $c$ the valuations of the eigenvalues of $\tilde T$ on
$U$ will be the same as those of $T$ on $U$,  namely $\nu_1,\dots,\nu_r$, and for
sufficiently small non-zero
$c$ the eigenvalues of $\tilde T$ on $W$ will be non-zero with valuation strictly
greater than $\nu_r$; when both of these things  happen, we will have 
\[
\nu_1+\dots + \nu_r=\tilde \nu_1+ \dots + \tilde \nu_r,
\] 
with equality actually holding term-by-term. 
Thus
we see that when
$c$ is sufficiently small and non-zero, $\tilde T$ satisfies all the hypotheses of
our theorem (for the given $\Lambda,U,W$), invertibility being ensured by
having no eigenvalue equal to $0$. Applying the theorem to $\tilde T$, we conclude
that 
\[
\Lambda=(\Lambda \cap U) \oplus (\Lambda \cap W),
\]
as desired. 
\end{proof}

\section{Recognizing elements in the subset $\mathcal {BT}_M$ of $\mathcal
{BT}_G$} \label{sec4}

\subsection{Notation}\label{r.1} 

Let $G$ be a split connected reductive group over $F$, and let $A$ be a
split maximal torus of~$G$ over~$F$.  We write $A_G$ for the split torus obtained
as the identity component of the center of $G$. We denote by
$\mathcal B=\mathcal B(A)$ the set of Borel subgroups of~$G$ containing $A$. For
$B=AU \in \mathcal B$ ($U$ being the unipotent radical of $B$) we denote by $\bar
B=A\bar U$ the Borel subgroup in
$\mathcal B$ that is opposite to~$B$. 

Let $M$ be a Levi subgroup of $G$ containing $A$.  We write
$\mathcal P(M)$ for the (finite) set of parabolic
$F$-subgroups of
$G$ admitting $M$ as Levi component.  For $P=MN \in \mathcal P(M)$, we write
$R_M$ (respectively, $R_N$)  for the set of roots of $A$ in $M$ (respectively,
$N$), and we write
$\bar P=M\bar N$ for the
 parabolic subgroup in $\mathcal P(M)$ that is opposite to $P$. We denote
by $\Lambda_M$ the quotient of $X_*(A)$ by the coroot lattice for $M$, and by
$p_M:X_*(A) \to \Lambda_M$ the canonical surjection.  Finally, we extend scalars
to 
$\mathbb R$, obtaining a linear map 
\[
\mathfrak a:=X_*(A)\otimes_\mathbb Z \mathbb R \to 
\Lambda_M \otimes_\mathbb Z \mathbb R=:\mathfrak a_M
\] 
which will still be denoted $p_M$. When $M=A$ we have $\mathfrak a_A=\mathfrak a$,
and in this case we will always suppress the subscript.  It is only in this
section of the paper that we will use the notation $\mathfrak a$ in this way; in
subsequent sections $\mathfrak a$ will always denote the Lie algebra of $A$.

\subsection{Review of Arthur's $(G,M)$-orthogonal sets}\label{subsec,gmorthog} 
 Let $M$ be a Levi subgroup containing
$A$.  
For adjacent $P=MN$, $P'=MN'$ in $ \mathcal P(M)$ we consider
the collection of elements in $\mathfrak a_M$ obtained as  images of
coroots
$\alpha^\vee$ with $\alpha \in R_N \cap R_{\bar N'}$.  Let 
$\beta_{P,P'}$ denote the unique member of this collection of which all other
members  are positive  multiples. In
case $M=A$, so that $P,P'$ are Borel subgroups, $\beta_{P,P'}$ is the unique
coroot of $A$ that is positive for $P$ and negative for $P'$.

Recall (see \cite{arthur76}) that a family  of points $x_P$ in $\mathfrak a_M$,
one for each
$P
\in \mathcal P(M)$, is said to be a \emph{$(G,M)$-orthogonal set}  if for
every pair
$P, P' \in \mathcal P(M)$ of adjacent parabolic subgroups 
there exists a real number  
$r$ (necessarily unique) such that
\[
x_{P}-x_{P'}=r\beta_{P,P'}.
\] 
  When all the numbers
$r$ are non-negative (respectively, non-positive), the $(G,M)$-orthogonal
set is said to be
\emph{positive} (respectively, \emph{negative}). Note that these concepts still
make sense when $\mathfrak a_M$ is replaced by its underlying affine space, since
the difference of two elements in that affine space is a well-defined element of
the vector space
$\mathfrak a_M$.

Let $P\in \mathcal P(M)$,
and  let $x=(x_B)_{B \in \mathcal B}$ be a $(G,A)$-orthogonal set. The map $B
\mapsto B \cap M$ identifies  ${\{B \in
\mathcal B(A):B \subset P\}}$   with $\mathcal B^M(A)$, the set of Borel subgroups
in $M$ containing $A$.  The points 
$(x_B)_{\{B \in
\mathcal B(A):B \subset P\}}$ form an $(M,A)$-orthogonal set and
thus have the same image, call it $y_P$, 
 in $\mathfrak a_M$. 
The family $y=(y_P)_{P \in
\mathcal P(M)}$   is then  a $(G,M)$-orthogonal set in $\mathfrak a_M$,
positive if
$(x_B)$ is, and is called the $(G,M)$-orthogonal set \emph{associated} to
$x$.  

\subsection{Recognizing positive $(G,A)$-orthogonal sets coming from
$M$}\label{sub.rpga} In the
next lemma, which proves the equivalence of five conditions on a positive
$(G,A)$-orthogonal set $x$, the main  point is that the seemingly  weak
condition (4) implies the rather strong condition (1). When these five equivalent
conditions  hold, we say that $x$ \emph{comes from
$M$}.

\begin{lemma}\label{lem.gmo}
Let $x=(x_B)_{ B \in \mathcal B}$ be a positive $(G,A)$-orthogonal set, and let
$y=(y_P)_{P \in \mathcal P(M)}$ be the positive
$(G,M)$-orthogonal set associated to $x$. Then the following five conditions on
$x$ are equivalent. 
\begin{enumerate}
\item There exists a positive $(M,A)$-orthogonal set $z=(z_{B_M})_{B_M \in \mathcal
B^M(A)}$ such that $x_B=z_{B\cap M}$ for all $B \in \mathcal B(A)$.  
\item $p_M(x_{B_1})=p_M(x_{B_2})$  for all
$B_1,B_2
\in
\mathcal B(A)$.  
\item $y_{P_1}=y_{P_2}$ for all $P_1,P_2 \in \mathcal P(M)$. 
\item There exists $P \in \mathcal P(M)$ such that $y_P=y_{\bar P}$. 
\item $x_{B'}=x_{B''}$ when $B',B''$ are adjacent and $\beta_{B',B''}$ is not a
coroot of $M$. 
\end{enumerate}
\end{lemma}
\begin{proof} 
It is clear that  (1) implies (2):  the points $z_{B_M}$ have the same
image in
$\mathfrak a_M$, so the same is true of the points $x_B$. 

Next we check that (2) implies (3). For $i=1,2$ choose $B_i \in \mathcal B$ such
that
$B_i
\subset P_i$. Then 
\[
y_{P_1}=p_M(x_{B_1})=p_M(x_{B_2})=y_{P_2}.
\] 

After noting that   
(3) clearly implies (4), we next  show that (4) implies (5).  
 Choose some $B \in \mathcal B$ such that $B \subset
P$. (This  amounts to choosing an element in $\mathcal B^M(A)$.) It
then follows that $\bar B \subset \bar P$. Note that 
\[
p_M(x_B)=y_P=y_{\bar P}=p_M(x_{\bar B}),
\]
which proves that 
\[
p_M(x_B)=p_M(x_{\bar B}).
\]
On $\mathfrak a_M$ we have the usual partial order determined by $P=MN$: in this
partial order we have $y \ge y'$ iff $y-y'$ is a non-negative linear combination
of elements of the form $p_M(\alpha^{\vee})$ for some root $\alpha \in R_N$. 

Let $B=B_0,B_1,\dots,B_l=\bar B$ be any minimal gallery from $B$ to $\bar B$. Then 
\[
x_B-x_{\bar B}= \sum_{i=1}^l (x_{B_{i-1}}-x_{B_i}).
\]
Combining the last two displayed equations, we find that 
\[
0=\sum_{i=1}^l p_M(x_{B_{i-1}}-x_{B_i}), 
\] 
and since each term in the sum on the right hand side of this equation is $\ge 0$
in our partial order on $\mathfrak a_M$, we conclude that each term is $0$.
Letting $\alpha_i$ be the unique root that is positive for $B_{i-1}$ and negative
for $B_i$, we have $x_{B_{i-1}}-x_{B_i}=r_i\alpha_i^\vee$ for some non-negative
number $r_i$, and we know that 
\[
0= p_M(x_{B_{i-1}}-x_{B_i})=r_ip_M(\alpha_i^\vee).
\] 
This yields no information when $\alpha_i \in R_M$, but when $\alpha_i \notin R_M$,
then $p_M(\alpha_i^\vee) \ne 0$, and therefore $r_i=0$, from which it follows that
$x_{B_{i-1}}=x_{B_i}$. 

We now know that (5) is true for certain adjacent pairs
$B',B''$, namely those of the form $B_{i-1},B_i$ for some $i$. To  establish
(5) fully we will now check that any adjacent pair
$B',B''$ arises from a suitable minimal gallery from $B$ to $\bar B$. (It turns
out that we do not need to vary our initial choice of $B$.) We write
$l(B,B')$ for the common length of all minimal galleries joining $B$ and $B'$ (and
will use parallel notation for other pairs of Borel subgroups). Then $l(B,B')$ and
$l(B,B'')$ differ by $1$, since $B',B''$ are adjacent. Exchanging $B'$ with $B''$
if necessary, we can find a positive integer $i$ such that $l(B,B')=i-1$ and
$l(B,B'')=i$. Note that $l(B'',\bar B)=l-i$, where $l$ is the number of
positive roots in the root system of $G$. Splicing together any two minimal
galleries
$B=B_0,\dots,B_{i-1}=B'$ and
$B''=B_i,B_{i+1},\dots,B_l=\bar B$, we
obtain a minimal gallery joining $B,\bar B$ and an integer $i$ such that
$(B',B'')=(B_{i-1},B_i)$, as desired. This completes the proof that (4) implies
(5). 

It remains only to prove that (5) implies (1).  Let $B_1,B_2 \in \mathcal B$ and
suppose that $B_1\cap M=B_2\cap M$. Equivalently, we are supposing that no root in
$M$ separates
$B_1$ from
$B_2$. Then by choosing a minimal gallery joining $B_1,B_2$, we see from (5) that
$x_{B_1}=x_{B_2}$. Thus there exists a unique family $z=(z_{B_M})_{B_M \in
\mathcal B^M(A)}$ such that $x_B=z_{B\cap M}$. 

We just need to check that $z$  is a positive $(M,A)$-orthogonal set. So suppose
that $B_M',B_M'' \in \mathcal B^M(A)$ are adjacent, and let $\alpha$ be the unique
root in $M$ which is positive for $B_M'$ and negative for $B_M''$. Choose $B',B''
\in \mathcal B$ such that $B'\cap M=B'_M$ and $B''\cap M=B''_M$, and then choose a
minimal gallery from $B'$ to $B''$. For exactly one adjacent pair in this gallery
the separating root will be $\alpha$, and for the other adjacent pairs the
separating root will not be a root of $M$. From (5) (and using that $x$ is a
positive $(G,A)$-orthogonal set) we conclude that 
\[
x_{B'}-x_{B''}=0+\dots+0+r\alpha^\vee + 0+\dots+0
\]
for some non-negative number $r$. In other words
$z_{B'_M}-z_{B''_M}=r\alpha^\vee$, so that $z$ is indeed a positive
$(M,A)$-orthogonal set. 
\end{proof}

\begin{remark}
The lemma we just proved has some obvious variants. First, a similar result
applies to negative $(G,A)$-orthogonal sets $x$, as one sees by applying the lemma
to the positive $(G,A)$-orthogonal set $-x$. Second, there is no harm in allowing
our $(G,A)$-orthogonal sets to take values in an affine
space under $\mathfrak a$ (for example the apartment of $A$ in the enlarged
building for $G$ is such an affine space and cannot be identified with
$\mathfrak a$ without choosing a basepoint in that apartment). 
\end{remark}

\subsection{Bruhat-Tits building $\mathcal{BT}_G$} 
We are going to work with the enlarged Bruhat-Tits building $\mathcal{BT}_G$ of
$G$, though we are going to drop the word ``enlarged.'' The building for $G$ is
the cartesian product of the buildings of $A_G$ and $G/A_G$, the building for
$A_G$ being the affine space underlying the real vector space $X_*(A_G)_\mathbb
R\simeq\mathfrak a_G$. The building
for
$G$ is canonical up to isomorphism, but not up to \emph{unique} isomorphism, as
$\mathfrak a_G$ acts on $\mathcal{BT}_G$, preserving all its natural structures
(see \cite[4.2.16]{bruhat-tits84} for a discussion of the precise sense in which
$\mathfrak a_G$ is the automorphism group of the building). 

The apartment of $A$ in $\mathcal{BT}_G$ serves as a building for $A$, and thus we
will denote it by $\mathcal{BT}_A$. More generally (see
\cite[5.1.3]{bruhat-tits84}) for any Levi subgroup $M$ of $G$ containing $A$, the
subset $M(F)\mathcal{BT}_A$ of $\mathcal{BT}_G$ serves as a building
$\mathcal{BT}_M$ for $M$. 

\subsection{Retractions}  
Consider a Levi subgroup $M$ containing $A$ as well as a parabolic
subgroup $P=MN$ with Levi component $M$. We claim that there is an
$M(F)$-equivariant retraction $r_P:\mathcal{BT}_G \to \mathcal{BT}_M$, 
characterized by the following two properties: 
\begin{enumerate}
\item the restriction of $r_P$ to $\mathcal{BT}_M$ is the identity map on
$\mathcal{BT}_M$, and 
\item $r_P(nx)=r_P(x)$ for all $n \in N(F)$, $x \in \mathcal{BT}_G,$  
\end{enumerate}
the point being that the composed map 
\begin{equation}\label{eq.retdef}
\mathcal{BT}_M \hookrightarrow \mathcal{BT}_G \twoheadrightarrow
N(F)\backslash\mathcal{BT}_G 
\end{equation}
is a bijection. When $M=A$, so that $P$ is a Borel subgroup of $G$, the retraction 
$r_P$ was introduced by Bruhat-Tits
in \cite[2.9]{bruhat-tits72}. 

We now check that the existence of retractions $r_B$ with respect to Borel
subgroups implies the existence of $r_P$.  For this we choose a Borel subgroup
$B=AU$ such that $A
\subset B
\subset P$, and we denote by
$B_M=AU_M$ the Borel subgroup of $M$ obtained as the intersection of $B$ with $M$;
of course
$U=NU_M$. 

To see that \eqref{eq.retdef} is surjective, we write $x \in
\mathcal{BT}_G$ as $x=u\lambda$ for $u \in U(F)$ and $\lambda \in \mathcal{BT}_A$. 
Decomposing $u$ as $nu_M$ with $n \in N(F)$ and  $u_M \in U_M(F)$, we see that
$x=n(u_M\lambda) \in N(F)\mathcal{BT}_M$. 

To see that \eqref{eq.retdef} is
injective, we consider $x,x' \in \mathcal{BT}_M$, and we assume that there exists
$n \in N(F)$ such that $x'=nx$. We must show that $x'=x$. Using the surjectivity
of \eqref{eq.retdef} for $(M,B_M)$, we may write $x=u_M\lambda$,
$x'=u_M'\lambda'$ for $u_M,u_M' \in U_M(F)$, $\lambda,\lambda' \in
\mathcal{BT}_A$. Therefore $u_M'\lambda'=nu_M\lambda$, and the injectivity of
\eqref{eq.retdef} for $(G,B)$ implies that $\lambda'=\lambda$ and that
$u_M'^{-1}nu_M$ lies in the stabilizer of $\lambda$ in $U(F)$, namely  
$U(F) \cap K_\lambda $. Here we have written $K_\lambda$ for the parahoric
subgroup of $G(F)$ determined by $\lambda$, and we used that $U(F)$ essentially
lives in the simply connected cover of the derived group, where parahoric
subgroups are actually stabilizers of points in the building. Now we are going to
use the decomposition 
\[
U(F)\cap K_\lambda=\bigl(N(F)\cap K_\lambda\bigr)\bigl(U_M(F) \cap
K_\lambda\bigr), 
\]
a consequence of the fact that Bruhat and Tits define the $\mathfrak o$-form 
$G_\lambda$ (the one such that $K_\lambda=G_\lambda(\mathfrak o)$) in such a way as
to be compatible with the factorization of $U$ as a product of root groups (see
3.8.1, 3.8.3 and 4.6.2 in \cite{bruhat-tits84}). Applying this fact to
$u_M'^{-1}nu_M$, we find that
$(u_M')^{-1}u_M
\in K_\lambda$, and hence that $x'=x$. 

\subsection{Obvious compatibilities among retractions}
Given a Borel
subgroup 
$B$ of $G$ containing $A$, we obtain  a Borel subgroup $B \cap M$ of $M$
containing $A$. We then have the following compatibility between the retraction
$r_B$ and its analog $r^M_{B\cap M}:\mathcal{BT}_M \to \mathcal{BT}_A$ for $M$:
\begin{equation}\label{eq.cptMG} 
r_B(x)=r^M_{B \cap M}(x)
\end{equation}
for any $x$ in the subset $\mathcal{BT}_M$ of $\mathcal{BT}_G$. 

In the special case that $B \subset P$, the compatibility \eqref{eq.cptMG} can be
generalized to  
\begin{equation}\label{eq.cptPG} 
r_B=r^M_{B \cap M}\circ r_P.
\end{equation}

\subsection{The  $(G,A)$-orthogonal set in $\mathcal{BT}_A$ determined by an
element in
$\mathcal {BT}_G$} 
 For fixed $x \in \mathcal{BT}_G$ we may let $B=AU$ vary through the set
$\mathcal B(A)$ of  Borel subgroups of $G$ containing $A$, thus obtaining a
family
$x_B:=r_B(x)$ of points in $\mathcal{BT}_A$, one for each $B \in \mathcal B(A)$
(cmp. \cite{arthur76,harish-chandra66b}). The standard method of reduction to the
case of
$SL_2$ shows that
$(x_B)$ is a  negative $(G,A)$-orthogonal set in $\mathcal{BT}_A$. That this
orthogonal set is negative rather than positive is due to the  convention made by
Bruhat and Tits,  that for any cocharacter $\mu$ of $A$, the element $\mu(\varpi)
\in A(F)$ acts on
$\mathcal{BT}_A$ by translation by the negative of $\mu$.  

\begin{proposition}
Let $x \in \mathcal{BT}_G$ and let $B=AU$, $B'=AU'$ be Borel subgroups of $G$
containing $A$. Then $r_B(x)=r_{B'}(x)$ if and only if $x \in \bigl((U(F) \cap
U'(F)\bigr)\mathcal{BT}_A$. 
\end{proposition} 

\begin{proof}
($\Longleftarrow$) Say $x = vx_0$ with $v \in U(F) \cap U'(F)$ and $x_0 \in
\mathcal{BT}_A$. Then
$r_B(x)=x_0=r_{B'}(x)$. 

($\Longrightarrow$) Consider a minimal gallery $B=B_0,B_1,\dots,B_r=B'$ (of Borel
subgroups $B_i=AU_i$) joining $B$ to $B'$. The proof uses induction on $r$. When
$r=0$, we just need to note that $U(F)\mathcal{BT}_A=\mathcal{BT}_G$. 

For $r>0$ we first note that the vectors $r_B(x)-r_{B_1}(x)$ and
$r_{B_1}(x)-r_{B'}(x)$ are both $\le 0$ with respect to the partial order on
$\mathfrak a$ obtained from $B$; since their sum is zero, they are individually
zero, which implies that $r_{B_1}(x)=r_{B_r}(x)$. By induction we conclude that $x$
lies in $\bigl((U_1(F) \cap
U'(F)\bigr)\mathcal{BT}_A$. 

Now $U_1\cap U'=(U\cap U')U_\alpha$, where $U_\alpha$ is the root subgroup for the
unique root $\alpha$ of $A$ in $G$ that is positive for $B$ and negative for
$B_1$. Thus $x=vu_\alpha x_0$ for some $v \in U(F) \cap U'(F)$, $u_\alpha  \in
U_\alpha(F)$, $x_0 \in \mathcal{BT}_A$. Put $y=u_\alpha x_0$; we just need to show
that $y \in \mathcal{BT}_A$. Let $M$ be the Levi subgroup of $G$ containing $A$
whose root system is $\{\alpha,-\alpha\}$; the derived group of $M$ is then
isomorphic to either $SL_2$ or $PGL_2$. In any case $y  \in \mathcal{BT}_M$, and
the compatibility \eqref{eq.cptMG} shows that the images of $y$ under the two
retractions $\mathcal{BT}_M \to \mathcal{BT}_A$ are equal. We conclude that $y \in
\mathcal{BT}_A$, as desired, by making use of the fact that for any $z$ in the
building for $SL_2$, the distance between the two retractions of $z$ into the
standard apartment is twice the distance from $z$ to that standard apartment. 
\end{proof} 

\begin{corollary}\label{cor.recog}
Let $x \in \mathcal{BT}_G$. Then $x \in \mathcal{BT}_M$ if and only if the
negative $(G,A)$-orthogonal set $(r_B(x))_{B \in \mathcal B(A)}$ comes from $M$
in the sense of the definition made at the beginning of subsection
\textup{\ref{sub.rpga}}. 
\end{corollary} 
\begin{proof}
($\Longrightarrow$)
It is clear from compatibility \eqref{eq.cptMG} that $(r_B(x))_{B \in \mathcal
B(A)}$ comes from $(r^M_{B_M}(x))_{B_M  \in \mathcal B^M(A)}$.  

 ($\Longleftarrow$) Suppose that $(r_B(x))_{B \in \mathcal
B(A)}$ comes from $M$.  Lemma \ref{lem.gmo} implies that
$r_{B_1}(x)=r_{B_2}(x)$ whenever $B_1,B_2 \in \mathcal B(A)$ are adjacent and
separated by a root for $G$ that is not a root for $M$. Choose a parabolic
subgroup $P=MN$ with Levi component $M$, and consider also the opposite parabolic
subgroup $\bar P=M\bar N$. Choose $B_M=AU_M \in \mathcal B^M(A)$. Then there are
unique $B=AU, B'=AU'\in \mathcal B(A)$ such that $B \subset P$, $ B' \subset
\bar P$ and $B \cap M=B_M=B'\cap M$. Note that $U \cap U'=U_M$. Since $B \cap
M=B'\cap M$, no root of
$G$ that separates $B$ and $B'$ is a root of $M$, and therefore $r_B(x)=r_{B'}(x)$.
By the previous proposition we conclude that \[x \in \bigl((U(F) \cap
U'(F)\bigr)\mathcal{BT}_A=U_M(F)\mathcal{BT}_A=\mathcal{BT}_M.
\]
\end{proof}

\begin{remark}\label{rmk.error}
The previous corollary is a generalization of Lemma 2.1 in \cite{kottwitz03}. It
should be noted that the proof of that lemma in \cite{kottwitz03} is slightly
wrong. The error occurs in the next to last sentence of the proof. 
\end{remark}

\section{Compatibility of $\mathfrak k_x$ with $M$ usually implies that
$x
\in \mathcal{BT}_M$}

We retain the notation of section \ref{sec4} (see \ref{r.1}). In particular $G$ is
a split connected reductive group over $F$ with split maximal torus $A$. In
addition we will make use of $\mathfrak g=\Lie G$, a Lie algebra over the ground
field
$F$. 

\subsection{Review of parahoric subalgebras of $\mathfrak g$}\label{sub.para} 
Since $A$ is split, it extends canonically to a smooth group scheme over
$\mathfrak o$, all of whose geometric fibers are tori, and in this way we obtain a
lattice $\mathfrak a(\mathfrak o)$ in $\mathfrak a$, namely the Lie algebra of the
group scheme $A$ over $\mathfrak o$. 

To any point $x \in \mathcal{BT}_G$ Bruhat and Tits \cite{bruhat-tits84} associate
a smooth group scheme
$G_x$ over $\mathfrak o$, and $K_x:=G_x(\mathfrak o)$ is parahoric subgroup of
$G(F)$. Actually Bruhat-Tits define a number of variants of $G_x$; we will use the
one for which the special fiber is connected and which is fabricated using the
$\mathfrak o$-form of $A$ we just discussed.  The Lie algebra
$\mathfrak k_x$ of $G_x$ is a parahoric subalgebra of
$\mathfrak g$. When it is
necessary to indicate which group we are working with, we write $\mathfrak k^G_x$
instead of $\mathfrak k_x$. 
When $x$ lies in the apartment of $A$, the parahoric
subalgebra
$\mathfrak k_x$ is compatible with the root space decomposition of $\mathfrak g$,
in the sense that it is the direct sum of $\mathfrak a(\mathfrak o)$ and its
intersections with the various root spaces in $\mathfrak g$, and for any Levi
subgroup $M$ containing $A$ we have    
 $\mathfrak k^M_x=\mathfrak k_x \cap \mathfrak m$,
where $\mathfrak m=\Lie M$, again a Lie algebra over $F$.  

\subsection{Notion of compatibility of $\mathfrak k_x$ with $M$} 
Let $M$ be a Levi subgroup of $G$ containing $A$, and let $A_M$ be identity
component of the center of $M$. Of course $A_M$ is a split torus, since $G$ is a
split group. We then have the canonical direct sum decomposition 
$\mathfrak g=\mathfrak m \oplus \mathfrak m^{\perp}$, where $\mathfrak m^{\perp}$
is by definition the direct sum of all the non-zero weight spaces for $A_M$ on
$\mathfrak g$. (When the Killing form is non-degenerate, $\mathfrak m^{\perp}$ is
actually the subspace perpendicular to $\mathfrak m$ under the Killing form, which
may help to explain the  notation we chose.) We write $q_M:\mathfrak g
\twoheadrightarrow \mathfrak m$ for the projection map obtained from the direct
sum decomposition $\mathfrak g=\mathfrak m \oplus \mathfrak m^{\perp}$; when it is
necessary to indicate which group we are working with, we write $q_M^G$ instead
of $q_M$. It is obvious that $q_A=q^M_A \circ q_M$. 

For $x \in \mathcal{BT}_G$ we say that $\mathfrak k_x$ is \emph{compatible} with
$M$ if $\mathfrak k_x$ is the direct sum of its intersections with $\mathfrak m$
and $\mathfrak m^{\perp}$, or, equivalently, if $\mathfrak k_x\cap \mathfrak
m=q_M(\mathfrak k_x)$. 
Most of the following lemma will be used only in the proof of the next theorem and
is of little independent interest. However the fourth part of the lemma will be
used again later. 

\begin{lemma}\label{lem.handy}
Let $P=MN$ be a parabolic subgroup with $M$ as Levi component. Let $x \in
\mathcal{BT}_G$ and put $y=r_P(x)$. Then the following six statements hold. 
\begin{enumerate} 
\item If $x$ lies in the subset $\mathcal{BT}_M$ of $\mathcal{BT}_G$, then
$\mathfrak k_x$ is compatible with $M$. 
\item $\mathfrak k^M_y=q_M(\mathfrak k_x \cap\mathfrak p)$.  
\item $\mathfrak k_x \cap \mathfrak m \subset \mathfrak k^M_y \subset
q_M(\mathfrak k_x)$. Thus $\mathfrak k_x$ is compatible with $M$ if and only if 
$\mathfrak k_x \cap \mathfrak m = \mathfrak k^M_y =
q_M(\mathfrak k_x)$.  
\item $\mathfrak k_x \cap \mathfrak a \subset \mathfrak k^M_y \cap \mathfrak a 
\subset
\mathfrak a(\mathfrak o) \subset q^M_A(\mathfrak k^M_y) \subset q_A(\mathfrak
k_x)$. 
\item If $\mathfrak k_x$ is compatible with $A$, then $\mathfrak k^M_y$ is
compatible with $A$. 
\item If $\mathfrak k_x$ is compatible with $M$ and $\mathfrak k^M_y$ is
compatible with $A$, then $\mathfrak k_x$ is compatible with $A$. 
\end{enumerate}
\end{lemma}
\begin{proof}
(1) Since the set of $x$ such that $\mathfrak k_x$ is compatible with $M$ is
stable under the action of $M(F)$ on $\mathcal{BT}_G$, it is enough to show that
$\mathfrak k_x$ is compatible with $M$ when $x \in \mathcal{BT}_A$, and this
follows  from the discussion at the end of subsection \ref{sub.para}.

(2)  Say $x=nm\lambda$ for $m \in M(F)$, $n \in N(F)$, $\lambda \in
\mathcal{BT}_A$. Then $y=m\lambda$ and both sides of the equality we are trying to
prove  equal $m\mathfrak k^M_\lambda m^{-1}$. Here we used that $\mathfrak p$
is stable under the adjoint action of $P(F)$, that the restriction of $q_M$ to
$\mathfrak p$ is
$P(F)$-equivariant, and that $q_M(\mathfrak k_\lambda \cap \mathfrak p)=\mathfrak
k^M_\lambda$, a consequence of the discussion at the end of subsection
\ref{sub.para}.

(3) The first statement follows from (2), and the second follows from the first. 

(4)  The first
inclusion in (4) comes from intersecting the first inclusion in (3) with
$\mathfrak a$.  The last inclusion in (4) comes from applying $q^M_A$ to the second
inclusion in (3). The middle two inclusions in (4) come from
applying (3) to
$(M,A,y)$ (rather than $(G,M,x)$). 

(5) Assume that $\mathfrak k_x$ is compatible with $A$. Then we have equality of
the first and last lattices in (4). Therefore we also have equality of the second
and fourth lattices in (4), which is to say that $\mathfrak k^M_y$ is compatible
with $A$. 

(6) When $\mathfrak k_x$ is compatible with $M$, the reasoning used in the proof
of (4) shows that the first and last inclusions in (4) are equalities. When
$\mathfrak k^M_y$ is compatible with $A$,  the middle two inclusions in (4) are
also equalities. Therefore all the lattices in (4) are equal, and $\mathfrak k_x$
is compatible with $A$. 
\end{proof}

\begin{theorem}\label{thm.comim}
Let $x \in
\mathcal{BT}_G$. If $x \in 
\mathcal{BT}_M$, then $\mathfrak k_x$ is compatible with $M$. The converse is true
under any of the following three assumptions. 
\begin{enumerate}
\item The residue field of $\mathfrak o$ is not of characteristic $2$. 
\item The center of $G$ is a torus.  This is stronger than the assumption that the
center of
$G$ is connected, since we  mean the center in the scheme theoretic sense;
for example in characteristic  $2$ the center of $SL_2$ is connected, but 
not a torus.   
\item The point $x$ is special in the sense that its retractions into $\mathcal
{BT}_A$ are special points in that apartment. We remind the reader that  $x$
is special if and only if every geometric fiber of the group scheme $G_x$ is a 
connected  reductive group. 
\end{enumerate} 
\end{theorem}
\begin{proof}
In the previous lemma it was shown that $\mathfrak k_x$ is compatible with $M$
when $x \in \mathcal {BT}_M$. Our real task is to prove the converse. In the first
part of this proof we will treat the special case in which $M=A$, and then we will
use this special case to handle the general one. 

Thus, for the moment we assume that $M=A$. First we examine the case in which $G$
has semisimple rank $1$, so that $G$ is isomorphic to the direct product of a
split torus and one of the three groups $SL_2$, $GL_2$, and $PGL_2$. It is
harmless to discard the torus factor, and thus we may as well assume that $G$ is
one of the three groups just mentioned, and that $M$ is the torus of diagonal
matrices in $G$. We will make use of the upper triangular Borel subgroup $P=MU$
and the corresponding positive root $\alpha$. 

We write $x$ as $x=u\lambda$ for $\lambda \in \mathcal {BT}_M$ and 
\begin{equation*} 
u=
\begin{bmatrix}
1 & t \\0 &1
\end{bmatrix} 
\in U(F).
\end{equation*}
We are going to work out concretely what it means for $\mathfrak k_x$ to be
compatible with $M$. 

The root vectors 
\begin{equation*} X_+:=
\begin{bmatrix}
0 & 1 \\0 &0
\end{bmatrix}, \ 
X_-:=\begin{bmatrix}
0 & 0 \\1 &0
\end{bmatrix}
\end{equation*} 
form a basis for $\mathfrak m^{\perp}$. There are unique integers $a,b$ such that 
\[
\mathfrak k_\lambda=\mathfrak m(\mathfrak o) \oplus \varpi^a\mathfrak o X_+
\oplus \varpi^b\mathfrak o X_-,
\] 
and moreover the sum  $a+b$ is $0$ (resp., $1$) if $x$ is special (resp., not
special). 

From the third part of Lemma \ref{lem.handy} it follows that $\mathfrak k_x$ is
compatible with $M$ if and only if 
\begin{equation}\label{eq.cM1}
\mathfrak m(\mathfrak o) \subset \mathfrak k_x
\end{equation}
and 
\begin{equation}\label{eq.cM2}
q_M(\mathfrak k_x) \subset \mathfrak m(\mathfrak o).
\end{equation}     

Since $\mathfrak m(\mathfrak o) \subset \mathfrak k_\lambda$, \eqref{eq.cM1} is
equivalent to $(\Ad(u^{-1})-1)\mathfrak m(\mathfrak o) \subset \mathfrak
k_\lambda$, which boils down to $\ad(-tX_+)\mathfrak m(\mathfrak o) \subset
\mathfrak k_\lambda$, or, equivalently, 
\begin{equation}\label{eq.cM1'}
t\alpha\bigl(\mathfrak m(\mathfrak o)\bigr) \subset \varpi^a\mathfrak o.
\end{equation}

Now  
\[
\mathfrak k_x=u\mathfrak k_\lambda u^{-1}=u\bigl(\varpi^b\mathfrak o X_- \oplus
(\mathfrak k_\lambda \cap \mathfrak p)  
\bigr)u^{-1} = \varpi^b\mathfrak o \Ad(u)(X_-) \oplus (\mathfrak k_x \cap
\mathfrak p)
\]
and since $q_M(k_x\cap \mathfrak p)=\mathfrak k^M_\lambda =\mathfrak m(\mathfrak
o)$ by the second part of Lemma \ref{lem.handy}, the condition
\eqref{eq.cM2} is equivalent to 
\[
q_M\bigl(\varpi^b\Ad(u)(X_-)\bigr) \in \mathfrak m(\mathfrak o),
\] 
and this boils down to the condition 
\begin{equation}\label{eq.cM2'}
t\varpi^bH_\alpha \in  \mathfrak m(\mathfrak o),
\end{equation}   
where $H_\alpha$ is the coroot for $\alpha$, viewed as an element of the Lie
algebra of $A=M$. 

Now we can complete the proof in the special case under consideration. Suppose that
$\mathfrak k_x$ is compatible with $M$. If $G$ is $GL_2$ or $PGL_2$, or if the
residual characteristic is not $2$, then $\alpha(\mathfrak m(\mathfrak
o))=\mathfrak o$, and \eqref{eq.cM1'} says that $t \in \varpi^a\mathfrak o$.
Therefore $u$ fixes $\lambda$, so that $x=u\lambda =\lambda \in \mathcal {BT}_M$.

When $G$ is $SL_2$, then \eqref{eq.cM1'} says only that $2t \in \varpi^a\mathfrak
o$, and when the residual characteristic is $2$, this is not enough to conclude
that $u$ fixes $\lambda$. However, if $x$ is special, then $b=-a$, so that
\eqref{eq.cM2'} becomes the statement that $t \in \varpi^a\mathfrak o$ (since for 
$SL_2$ we have $\mathfrak m(\mathfrak o)=\mathfrak o H_\alpha$), and again we
conclude that $x=u\lambda=\lambda \in \mathcal {BT}_M$. 

The next step is to prove the theorem for general $G$, but with $M$ still equal to
$A$. So we assume that $x \in \mathcal {BT}_G$ is such that $\mathfrak k_x$ is
compatible with $A$, and we must prove that $x \in \mathcal {BT}_A$. By Corollary
\ref{cor.recog} it is enough to check that $r_B(x)=r_{B'}(x)$ whenever $B,B'$ are
adjacent Borel subgroups containing $A$. Let $\alpha$ be the unique root of $A$
that is positive for $B$ and negative for $B'$,  let $M_\alpha$ be the Levi
subgroup of semisimple rank $1$  containing $A$ whose roots are
$\{\alpha,-\alpha\}$, and let $P=M_\alpha N$ be the unique parabolic subgroup with
Levi component $M_\alpha$ that contains both $B$ and $B'$. Put $y=r_P(x)$.

From the fifth part of Lemma \ref{lem.handy} it follows that $\mathfrak
k_y^{M_\alpha}$ is compatible with $A$. If the center of $G$ is a torus, the same
is true for any Levi subgroup of $G$, and in particular this is so for $M_\alpha$.
Also, if $x$ is special, so too is $y$. Therefore, from the semisimple rank $1$
case that has already been treated, we conclude that $y \in \mathcal {BT}_A$ and
hence that 
\[
r_B(x)=r_{B \cap M_\alpha}(y)=y=r_{B'\cap M_\alpha}(y)=r_{B'}(x),
\]
as desired. 

Now consider the general case. We are given $x \in\mathcal {BT}_G$ such that
$\mathfrak k_x$ is compatible with $M$, and we must show that $x \in \mathcal
{BT}_M$. Choose a parabolic subgroup $P=MN$ with Levi component $M$, and put
$y=r_P(x) \in \mathcal {BT}_M$.  Since it is harmless to multiply $x$ on the left
by any element of $M(F)$, we may as well assume that $y \in \mathcal {BT}_A$. To
prove the theorem it suffices to show that $x \in \mathcal {BT}_A$. 

Thus, by what we have already proved, it is enough to show that $\mathfrak k_x$ is
compatible with $A$, and this follows from the sixth part of Lemma \ref{lem.handy}
because $\mathfrak k_x$ is compatible with $M$ and $\mathfrak k^M_y$ is compatible
with $A$ (since $y \in \mathcal {BT}_A$). 
\end{proof}

\section{Review of root valuation functions and strata}\label{sec7}
\subsection{Notation}\label{sub.6.1}
For the rest of this paper we work with 
 $F=\mathbb C((\epsilon))$, $\mathcal O=\mathbb C[[\epsilon]]$. In addition we
consider a split connected reductive group $G$ over $\mathbb C$ and a split maximal
torus
$A$ in
$G$. Associated to $A$ we have its set $R$ of roots, its Weyl group $W$ and 
its Lie algebra  $\mathfrak a$. We remind the reader that it was only in section
\ref{sec4} that we used $\mathfrak a$ to denote $X_*(A)\otimes \mathbb R$. 
Moreover we have the root space decomposition 
\begin{equation}
\mathfrak g=\mathfrak a \oplus \bigl(\bigoplus_{\alpha \in R} \mathfrak
g_\alpha\bigr),
\end{equation}
where  $\mathfrak
g_\alpha$ is the root space corresponding to $\alpha \in R$. 
Finally, for $\alpha \in R$ and $n\in
\mathbb Z$ we will often use
$P_\alpha^n$ as a convenient abbreviation for $\epsilon^n\mathfrak
g_\alpha(\mathcal O)$.  

\subsection{Root valuation functions} Let $u \in \mathfrak a(F)$ and
assume that $u$ is regular  in $\mathfrak g(F)$. Then, as in \cite{gkm}, we obtain
from $u$ the function $r_u:R \to \mathbb Z$ defined by $r_u(\alpha):=\val
\alpha(u)$. 

Given any function $r:R \to \mathbb Z$, we denote by $\mathfrak
a(F)_r$ the subset of $\mathfrak a(F)$ consisting of all regular elements $u$ for
which $r_u=r$. In Proposition 3.4.1 of \cite{gkm} it was shown that $\mathfrak
a(F)_r$ is nonempty if and only if $r$ satisfies the following condition on the
subsets  $R_m:=\{\alpha \in R: r(\alpha) \ge m\}$ of $R$: for
every  $m \in \mathbb Z$ the subset $R_m$ is $\mathbb Q$-closed (equivalently, is
the root system of a Levi subgroup of $G$ containing $A$). In this paper we will
refer to functions $r$ satisfying this condition as \emph{root valuation
functions}.  

Now let $u$ be a regular semisimple element of $\mathfrak g(F)$. We  say that
$u$ is \emph{split} if its centralizer in $G$ is a split maximal $F$-torus in
$G$. Any split regular semisimple $u \in \mathfrak g(F)$ is $G(F)$-conjugate to an
element $u' \in \mathfrak a(F)$, well-defined up to the action of $W$ on
$\mathfrak a(F)$. The root valuation function $r_{u'}$ is then well-defined up to
the action of $W$. Turning this around, as in \cite{gkm}, we consider a root
valuation function $r:R \to \mathbb Z$ and then denote by $\mathfrak g(F)_r$ the
subset of
$\mathfrak g(F)$ consisting of split regular semisimple elements $u$ for which the
root valuation function  $r_{u'}$ lies in the $W$-orbit of  $r$. The subset
$\mathfrak g(F)_r$ is referred to as the
\emph{root valuation stratum} in $\mathfrak g(F)$ associated to $r$.

\section{Generalized affine Springer fibers adapted to a given root valuation
function} \label{sec8}

As before, for any $x$ in the building of the $F$-group  $G_F$, we
can consider the corresponding parahoric subgroup $K_x$ and its Lie algebra
$\mathfrak k_x$. 

\subsection{Goal} 
Our goal in this section is to investigate a new kind of affine Springer theory
that is adapted to a given root valuation stratum in $\mathfrak g(F)$. Let us then
fix a root valuation function
$r:R \to
\mathbb Z$ and a point $x$ in the apartment of $A$. Using $r,x$ we will construct
an ind-scheme
$Y_{r,x}$ and a morphism
$Y_{r,x} \to \mathfrak g(F)$, generalizing affine Springer theory for the
partial affine flag manifold $G(F)/K_x$. Over the particular root valuation
stratum $\mathfrak g(F)_r$ the fibers of this morphism are $0$-dimensional and
have a very simple description, as we will see in Theorem \ref{thm.fibers}.

\subsection{Definition of $\tilde r$ } We will be 
making use of the same concepts that arose in our discussion (see section
\ref{sec3}) of the linear Hodge-Newton decomposition for linear transformations.
The linear transformations of interest are the ones of the form $\ad(u)$ for $u
\in \mathfrak g(F)$. Let $d$ denote the dimension of $\mathfrak g$. Then the
Newton point for any linear transformation
$\mathfrak g(F) \to \mathfrak g(F)$ lies in $\tilde{\mathbb Q} _+^d$,
while the Hodge point of such a linear transformation with respect to some
lattice in $\mathfrak g(F)$ lies in the intersection of $(\mathbb Z \cup
\{\infty\})^d$ with $\tilde{\mathbb Q} _+^d$. 

We now use $r$ to produce an element $\tilde r \in (\mathbb Z \cup \{\infty\})^d
\cap \tilde{\mathbb Q} _+^d$. One way to give the definition is to say that
$\tilde r$ is the Newton point of
$\ad(u)$ for any $u$ in the root valuation stratum $\mathfrak g(F)_r$. More
concretely, $\tilde r$ is obtained by forming the $d$-tuple whose first $|R|$
entries are the integers $r(\alpha)$, listed in non-decreasing order, and whose
last $d-|R|=\dim A$ entries are all $\infty$. 

\subsection{Definition of $Y_{r,x} \to \mathfrak g(F)$} As we just said, in the
situation we are in, we may consider the Hodge point $\mu(\ad(u),\Lambda) \in
 \tilde{\mathbb Q} _+^d$ for any lattice $\Lambda$ in $\mathfrak g(F)$. The
lattices we care about are  the ones obtained as $\mathfrak
k_{gx}=g\mathfrak k_x g^{-1}$ for some $g \in G(F)/K_x$, and we define $Y_{r,x}$
to be the set of pairs $(u,g) \in \mathfrak g(F) \times (G(F)/K_x)$ satisfying the
condition that 
\[
\mu(\ad(u),\mathfrak k_{gx}) \le \tilde r.
\]
We also consider the map $\pi_{r,x}:Y_{r,x} \to \mathfrak g(F)$ defined by $(u,g)
\mapsto u$. 

The fibers of $\pi_{r,x}$ are generalizations of affine Springer fibers; indeed,
usual affine Springer theory is recovered by taking $r$ to be identically $0$.  It
follows from Theorem
\ref{VTlambda}(1) that if
$(u,g)
\in Y_{r,x}$, then 
\begin{equation}\label{eq.chainnu} 
\nu(\ad(u)) \le \mu(\ad(u),\mathfrak k_{gx}) \le \tilde r.
\end{equation} 
In particular, the fiber of $\pi_{r,x}$ over $u \in \mathfrak
g(F)$ is empty unless $\nu(\ad(u)) \le \tilde r$. 

\subsection{Fibers of $\pi_{r,x}$ over certain strata $\mathfrak a(F)_{r'}$} 
Let $r':R \to \mathbb Z$ be another root valuation function. We say that $r'$ is
\emph{weakly equivalent} to $r$ if $\tilde r'=\tilde r$. Of course if $r'$ is of
the form $wr$ for some $w \in W$, then $r'$ is weakly equivalent to $r$.

\begin{theorem}\label{thm.fibers}  
Suppose that $r'$ is a root valuation function that is weakly equivalent to $r$,
and let
$u \in
\mathfrak a(F)_{r'}$. Then the fiber of
$\pi_{x,r}$ over
$u$ consists of the set of 
$g
\in G(F)/K_x$ such that  $gx$ lies in the apartment
of $A$. For example, when $x$ lies in the interior of an alcove, so that $K_x$ is
an Iwahori subgroup, this fiber can be identified with the extended affine Weyl
group
$\tilde W$. 
\end{theorem}
\begin{proof}
It is obvious that when $gx$ lies in the apartment of $A$, the Hodge point
$\mu(\ad(u),\mathfrak k_{gx})$ is $\tilde r$, so such points $g$ do lie in the
fiber over $u$. 

Conversely, suppose that $g$ is a point in the fiber over $u$. Since $\nu(\ad
u)=\tilde r$,   we conclude from 
\eqref{eq.chainnu} that $\nu(\ad(u))=\mu(\ad(u),\mathfrak k_{gx})$. 
It then follows from Theorem \ref{VTlambda}(2), applied to the subspaces
$W=\mathfrak a(F)$ and $U=\mathfrak a^{\perp}(F)$, that $\mathfrak k_{gx}$ is
compatible with
$A$. From Theorem \ref{thm.comim} we  see that $gx$ lies in the apartment of
$A$, as desired. 
\end{proof}

\section{Admissible proalgebraic subgroups of $G(F)$}\label{sec9}
\subsection{Basic definitions} 
As before, for any point $x$ in the building of $G(F)$ we denote by $G_x$ the
group  scheme over $\mathcal O$ that Bruhat-Tits associate to $x$. As usual  
$K_x:=G_x(\mathcal O)$ is a parahoric subgroup of $G(F)$, and the Lie algebra of
$G_x$ is a parahoric subalgebra $\mathfrak k_x$ of $\mathfrak g(F)$. For any
nonnegative integer $n$ we denote by $K_{x,n}$ the kernel of $G_x(\mathcal O)
\twoheadrightarrow G_x(\mathcal O/P^n)$. Then the equality 
\[
K_x=\varprojlim_{n}K_x/K_{x,n}
\]
exhibits $K_x$ as a proalgebraic group over $\mathbb C$. 
Let $y$ be another point in the building. Then $K_x \cap K_y$ is a proalgebraic
subgroup of both $K_x$ and $K_y$.  

Let $K$ be a subgroup of $G(F)$. We say that $K$ is an \emph{admissible
proalgebraic subgroup} of $G(F)$ if the following two conditions hold: 
\begin{enumerate}
\item there exists $x$ in the building and $m \ge 0$ such that 
\begin{equation}\label{eq.AdPro}
K_x \supset K \supset K_{x,m},
\end{equation}
\item 
for one (equivalently, every) $x,m$ satisfying \eqref{eq.AdPro}, $K/K_{x,m}$ is a
closed algebraic subgroup of $K_x/K_{x,m}$. 
\end{enumerate}

The proalgebraic structures on $K$ inherited from the various   $K_x$ containing
$K$  all agree with each other, so that $K$ becomes a proalgebraic group in a
canonical way.

 \subsection{Maximal tori in admissible proalgebraic subgroups}

An admissible pro\-algebraic subgroup $K$ of $G(F)$ is a rather special kind of
proalgebraic group, in that its quotient by its prounipotent radical is a reductive
algebraic group. Therefore $K$ has maximal $\mathbb C$-tori, and any two such are
conjugate under $K$; here one needs to bear in mind that if 
\[
1 \to U \to H \to H/U \to 1
\]
is a short exact sequence of linear algebraic groups, with $U$ unipotent, then any
two maximal tori in $H$ having the same image in $H/U$ are conjugate under $U$
(not just under $H$). To prove this last fact one can apply the fourth part of
Theorem (10.6) in \cite{LAG} to the connected solvable group obtained as the
preimage in
$H$ of the common image in $H/U$ of the two maximal tori in question.

\subsection{Review of some results of Bruhat-Tits} 
We begin by reviewing  Bruhat-Tits' notion of concave function on the root system 
$R$. 

\begin{definition}
A function $f:R \to \mathbb R$ is said to be \emph{concave} if 
\begin{enumerate}
\item  $f(\alpha)+f(\beta) \ge f(\alpha +\beta)$ whenever $\alpha,\beta \in R$ are
such that $\alpha +\beta \in R$, and 
\item $f(\alpha)+f(-\alpha) \ge 0$ for all $\alpha \in R$.  
\end{enumerate} 
\end{definition}

The next proposition is a less general version of Bruhat-Tits'
\cite{bruhat-tits72}  Proposition (6.4.6).

\begin{proposition}
Let $f$ be a concave function on the root system $R$. Then there exists $x$
in the apartment $X_*(A)\otimes \mathbb R$ of $A$ such that 
\[
\alpha(x) \le f(\alpha)
\] 
 for all $\alpha \in R$.
\end{proposition}

Before stating the next result, we  remind the reader that for $m \in
\mathbb Z$ the lattice 
$P^m_\alpha$  in $\mathfrak
g_\alpha(F)$ was defined in subsection \ref{sub.6.1}.

\begin{corollary}  \label{cor.bt}
Let $\mathfrak k$ be a
lattice in
$\mathfrak g(F)$ of the form 
\[
\mathfrak k=\mathfrak a(\mathcal O) \oplus \bigl(\bigoplus_{\alpha \in
R}P_\alpha^{k(\alpha)}\bigr)
\]
for some function $k:R \to \mathbb Z$. Then $\mathfrak k$ is a subalgebra of
the Lie algebra $\mathfrak g(F)$ if and only if $k$ is a concave function, and in
this case there exists a point $x$ in the apartment of $A$ such that $\mathfrak k$
is contained in the parahoric subalgebra of $\mathfrak g(F)$ determined by $x$. 
\end{corollary}
\begin{proof}
We may as well assume that $G$ is semisimple. An easy calculation shows that
$\mathfrak k$ is closed under bracket if and only if $k$ is a concave function, 
and in this case the proposition above says that there exists $x$ in the
apartment of $A$ such that $\alpha(x)
\le k(\alpha)$ for all $\alpha \in R$, so that $\mathfrak k$ is contained in
the parahoric subalgebra  
\[\mathfrak k_x=\mathfrak a(\mathcal O) \oplus \bigl(\bigoplus_{\alpha \in
R}P_\alpha^{\lceil\alpha(x)\rceil}\bigr).
\]
\end{proof}

\begin{remark} We retain the notation in the corollary. 
When $k$ is concave, Bruhat-Tits (see  Thm.~4.5.4
and sections 4.6.4 and 4.6.8 of \cite{bruhat-tits84}) construct a smooth
affine group scheme $G_k$ over $\mathcal O$ with generic fiber $G_F$ and connected
special fiber, having the property that $\Lie G_k=\mathfrak k$; thus $G_k(\mathcal
O)$ is  a connected admissible proalgebraic subgroup of
$G(F)$ having Lie algebra  $\mathfrak k$. 
\end{remark}

\subsection{Recognizing maximal tori in the presence of a suitable torus action} 
We now interrupt our discussion in order to prove the following proposition, which
will be applied in the next subsection. In it we may as well work over any
algebraically closed field. 
\begin{proposition}\label{prop.rmtta}
Let $H$ be a linear algebraic group, equipped with the
action of a torus $A$, which we use to form the semidirect product $A \ltimes H$.
Let
$T$ be a torus in $H$ satisfying the following two conditions: 
\begin{enumerate}
\item $A$ centralizes $T$, so that $A\times T$ is a subgroup of  $A \ltimes H$,
and 
\item The identity component of the group $H^{A \times T}$ of fixed points of the
conjugation action of
$A\times T$ on $H$ is equal to $T$.
\end{enumerate} 
Then $T$ is a maximal torus in $H$. 
\end{proposition}
\begin{proof}
The identity component of the centralizer in $A\ltimes H$ of $A \times T$ is equal
to
$A
\times T$, and therefore $A \times T$ is a maximal torus in $A \ltimes H$. Let
$S$ be any torus in $H$ that contains $T$; we must show that $S =T$. Now some
conjugate in $A \ltimes H$ of the maximal torus $A \times T$ must contain $S$.
Since $A$ centralizes $T$, any such conjugate has the form $h^{-1}(A\times T)h$
for some $h
\in H$. Then $hSh^{-1} \subset T$. Therefore $S$ and $T$ have the same dimension,
and since $S$ contains $T$, this shows that $S=T$, as desired. 
\end{proof}

\subsection{An application} \label{sub.appl}
Now we return to our usual setup. We are going to use the proposition we just
proved in order to recognize  maximal tori in certain proalgebraic subgroups of
$G(F)$ that will arise later when we study root valuation lattices in $\mathfrak
g(F)$. 

We begin by recalling the usual action of $\mathbb C^\times$ on the field $F$ (by
field automorphisms). For this action an element $t \in \mathbb C^\times$ acts on
$\epsilon$ by multiplication by $t$, hence acts on the Laurent
power series $\sum_ia_i\epsilon^i$ by sending it to $\sum_i a_it^i\epsilon^i$. The
fixed field of this action is of course $\mathbb C$. This action then induces
actions of $\mathbb C^\times$ on $G(F)$ (having fixed points
$G(\mathbb C)$) and on $\mathfrak g(F)$ (having fixed points $\mathfrak g(\mathbb
C)$). 

We will be interested in $\mathbb C^\times$-stable sublattices in $\mathfrak
a(\mathcal O)$. Any such lattice is obtained in the following way. Let $V_0
\subset V_1 \subset V_2 \subset\dots$ be an increasing chain of linear subspaces
of $\mathfrak a(\mathbb C)$ such that $V_i=\mathfrak a(\mathbb C)$ for $i \gg 0$.
Then 
\[
\bigl\{\sum_{i=0}^\infty v_i\epsilon^i: v_i \in V_i \quad \forall i\ge0\bigr\}
\]
is a $\mathbb C^\times$-stable lattice in $\mathfrak a(\mathcal O)$. 

Now
$\mathfrak a(\mathcal O)$ is the Lie algebra of the proalgebraic group $A(\mathcal
O)$. Again consider a lattice in $\mathfrak a(\mathcal O)$. We say that it is
algebraic if it arises as the Lie algebra of a proalgebraic subgroup of
$A(\mathcal O)$ containing $\ker[A(\mathcal O) \twoheadrightarrow A(\mathcal
O/P^N)]$ for some sufficiently large $N$. The commutative proalgebraic group
$A(\mathcal O)$ is the cartesian product of the torus $A(\mathbb C)$ and the
prounipotent commutative algebraic group $\ker[A(\mathcal O) \twoheadrightarrow 
A(\mathbb C)]$. Thus any proalgebraic subgroup of $A(\mathcal O)$ of the type we
are considering is the cartesian product of $S(\mathbb C)$, $S$ being  some
subtorus of $A$, and a proalgebraic subgroup of $\ker[A(\mathcal O)
\twoheadrightarrow  A(\mathbb C)]$.
The corresponding algebraic lattice in $\mathfrak a(\mathcal O)$ is then the direct
sum of
$\mathfrak s(\mathbb C)$ and some lattice in $\epsilon \mathfrak a(\mathcal O)$. 

Now consider an $\mathcal O$-subalgebra $\Lambda$ in
$\mathfrak g(F)$ of the form 
\[
\Lambda=\Lambda_A \oplus \bigl(\bigoplus_{\alpha \in R}
P_\alpha^{\lambda(\alpha)}\bigr)
\]
for some family of integers $\lambda(\alpha)$ ($\alpha \in R$)  and a $\mathbb
C^\times$-stable algebraic lattice
$\Lambda_A$ in $\mathfrak a(\mathcal O)$. Thus there is a subtorus $S \subset A$
and an increasing chain $\mathfrak s(\mathbb C)=V_0 \subset V_1 \subset V_2
\subset \dots$ of linear subspaces of $\mathfrak a(\mathbb C)$ such $V_i=\mathfrak
a(\mathbb C)$ for
$i\gg 0$ and 
\[\Lambda_A=
\bigl\{\sum_{i=0}^\infty v_i\epsilon^i: v_i \in V_i \quad \forall i\ge0\bigr\}.
\]

Since $\Lambda$ is normalized by $A(\mathcal O)$, the lattice 
\[
\Lambda^\sharp=\mathfrak a(\mathcal O) \oplus \bigl(\bigoplus_{\alpha \in R}
P_\alpha^{\lambda(\alpha)}\bigr)
\] 
is also an $\mathcal O$-subalgebra of $\mathfrak g(F)$. It follows that $\lambda$
is a concave function on $R$, and hence that, by the result of Bruhat-Tits we just
reviewed, there exists a point $x$ in the apartment of $A$ such that
$\Lambda^\sharp$ is contained in the parahoric subalgebra $\mathfrak k_x$
associated to $x$. Thus we also have
$\Lambda
\subset \mathfrak k_x$. Since $\Lambda_A$ is an algebraic subalgebra of $\mathfrak
k_x$, and  each $P_\alpha^{\lambda(\alpha)}$ arises in an obvious way as the Lie
algebra of a proalgebraic subgroup of $K_x$, we conclude from Corollary (7.7) of
\cite{LAG} that 
$\Lambda$ is an algebraic subalgebra of $\mathfrak k_x$, in the sense that
there is a connected admissible proalgebraic subgroup
$\tilde\Lambda$ of the parahoric subgroup $K_x$ such that  Lie algebra of $\tilde
\Lambda$ is
$\Lambda$.

We are now going to use  Proposition \ref{prop.rmtta} to show that $S(\mathbb
C)$ is a maximal torus of $\tilde \Lambda$. Indeed, by that proposition we just
need   to find a torus acting on $\tilde \Lambda$ such that the induced action on
$\Lambda$ has fixed points
$\mathfrak s(\mathbb C)$. The right torus to use is $A(\mathbb C)\times \mathbb
C^\times$, with
$A(\mathbb C)$ acting by conjugation and $\mathbb C^\times$ acting as described
above. Since $\Lambda$ is  $\mathbb C^\times$-stable, so too is $\tilde \Lambda$. 
The fixed points of
$A(\mathbb C)$ on
$\Lambda$ are 
$\Lambda_A$, and the fixed points of $\mathbb C^\times$ on $\Lambda$ are  $\Lambda
\cap \mathfrak g(\mathbb C)$; therefore the fixed points of $A(\mathbb C)\times
\mathbb C^\times$ on $\Lambda$ are $\Lambda_A \cap \mathfrak g(\mathbb
C)=\mathfrak s(\mathbb C)$, as desired.

\section{Topological Jordan decomposition}\label{sec10}
\subsection{Review of the topological Jordan decomposition}\label{sub.tJd}   
The reader who is already familiar with the topological Jordan decomposition
should skip this subsection. 

Let $K$ be an admissible proalgebraic subgroup of $G(F)$. Choose $x,m$ such that
\eqref{eq.AdPro} holds. The Lie algebra $\mathfrak k$ of $K$ is then a Lie
$\mathbb C$-subalgebra of $\mathfrak g(F)$ such that $\mathfrak k_x \supset
\mathfrak k \supset \epsilon^m\mathfrak k_{x}$. Moreover we have 
\[
\mathfrak k=\varprojlim_{n \ge m} \Lie(K/K_{x,n})=\varprojlim_{n \ge m} \mathfrak
k/\epsilon^n\mathfrak k_x. 
\] 

Say that $u \in \mathfrak k$ is $\mathbb C$-semisimple (respectively,
topologically nilpotent) if its image in the Lie algebra of $K/K_{x,n}$ is
semisimple (respectively, nilpotent) for all $n \ge m$. For $u \in \mathfrak k$
the Jordan decompositions of the images of $u$ in $\Lie(K/K_{x,n})$ are compatible
with each other as $n$ varies, so that there exist unique $u_s,u_n \in\mathfrak k$
such that $u=u_s+u_n$, $[u_s,u_n]=0$, $u_s$ is $\mathbb C$-semisimple, and $u_n$
is topologically nilpotent. This is customarily called the topological Jordan
decomposition of
$u$ and is independent of the choice of $x,m$ for which 
\eqref{eq.AdPro} holds. 

Suppose that $K'$ is an admissible proalgebraic subgroup containing $K$. Then the
topological Jordan decompositions of $u \in \mathfrak k \subset \mathfrak k'$
coincide. In particular the topological Jordan decomposition of $u \in \mathfrak
k$ can be calculated inside the Lie algebra of any parahoric subgroup containing
$K$. 

An element $u \in \mathfrak g(F)$ is said to be \emph{integral} if it is contained
in some parahoric subalgebra $\mathfrak k_x$. Any
integral element $u \in \mathfrak g(F)$ has a topological Jordan decomposition
$u=u_s+u_n$, independent of the choice of parahoric subalgebra containing it. 

\subsection{Relation between root valuations and the topological Jordan
decomposition} 
Let $u \in \mathfrak a(\mathcal O)$ and assume that $u$ is regular, so that
$\alpha(u)\ne 0$ for all $\alpha \in R$. We then have the root valuation function
$r:R \to \mathbb Z_{\ge 0}$ determined by $u$, namely $r(\alpha):=\val \alpha(u)$.
Now $u$ is integral and its $\mathbb C$-semisimple part $u_s$ is simply the image
of $u$ under $\mathfrak a(\mathcal O) \twoheadrightarrow \mathfrak a(\mathbb C)
\hookrightarrow \mathfrak a(\mathcal O)$. It follows that 
\[
\{ \alpha \in R : r(\alpha) \ge 1\}=\{\alpha \in R : \alpha(u_s)=0\}.
\]
More generally, let $u$ be a regular semisimple element of $\mathfrak g(F)$ that
is $G(F)$-conjugate to some element in $\mathfrak a(\mathcal O)$. Choosing $g \in
G(F)$ such that $u_1=gug^{-1}$ lies in $\mathfrak a(\mathcal O)$, we obtain the
root valuation function $r(\alpha)=\val \alpha(u_1)$ for $u_1$. If we make a
different choice of $g$, we will change $u_1$ and $r$ by an element in the Weyl
group, but the cardinality of $\{\alpha \in R:r(\alpha) \ge 1\}$ will be
unchanged. 

Since $u$ is integral, we may consider its $\mathbb C$-semisimple part $u_s$. There
exists $h \in G(F)$ such that $v:=hu_sh^{-1}$ lies in $\mathfrak a(\mathbb C)$. Of
course we can take $h=g$, we we are not obliged to. Again $v$ is only well-defined
up to the action of the Weyl group, but the cardinality of $\{\alpha \in R:
\alpha(v)=0\}$ is independent of the choice of $h$. Moreover, taking
$h=g$, we see that 
\begin{equation}\label{eq.cardJ}
|\{\alpha \in R: r(\alpha) \ge 1\}|=|\{\alpha \in R:\alpha(v)=0\}|,
\end{equation}
a simple observation that will play an important role in the proof of the Key Lemma
\
\ref{lem.keyconj} needed to prove the  Conjugation Theorem \ref{thm.conj} for root
valuation lattices.

\section{Definition and first properties of root
valuation lattices}\label{sec.defrvl}
 We now assume
that $G$ is semisimple, not just reductive. Throughout this section we fix some
root valuation function
$r:R
\to
\mathbb Z$.  We have already discussed the subset $\mathfrak a(F)_r$. We will also
make use of the lattice   
\begin{equation}
\mathfrak a(F)_{\ge r}:=\{u \in \mathfrak a(F):
\val\alpha(u)\ge r(\alpha)\quad \forall \alpha \in R\} 
\end{equation}
in $\mathfrak a(F)$ obtained as the closure of $\mathfrak
a(F)_r$.

\subsection{$A(\mathcal O)$-stable lattices in $\mathfrak g(F)$} 
A lattice $\Lambda$ in $\mathfrak g(F)$ is $A(\mathcal O)$-stable for the
adjoint action of $A(\mathcal O)$ on $\mathfrak g(F)$ if and only if it is
of the form 
\begin{equation}
\Lambda_A \oplus \bigl( \bigoplus_{\alpha \in R}
P_\alpha^{\lambda(\alpha)}\bigr)
\end{equation}
for some lattice $\Lambda_A$ in $\mathfrak a(F)$ and integers
$\lambda(\alpha)$. 

\subsection{Root valuation lattices} 
We are interested in $A(\mathcal O)$-stable lattices in $\mathfrak g(F)$
 adapted to studying the root valuation stratum  $\mathfrak
g(F)_r$. Consider then 
a  function $\lambda:R \to \mathbb Z$. To $r,\lambda$ we associate the
$A(\mathcal O)$-stable lattice $\Lambda$ defined by
\begin{equation}
\Lambda:=\mathfrak a(F)_{\ge r} \oplus \bigl( \bigoplus_{\alpha \in
R}P_\alpha^{\lambda(\alpha)}\bigr).
\end{equation}
When we wish to remember the particular $r,\lambda$ used to form this
lattice we write $\Lambda_{r,\lambda}$ rather than $\Lambda$. 

Now we seek a condition on $\lambda$  guaranteeing that a generic
element in $\Lambda$ is conjugate to  an element in $\mathfrak a(F)_r$.
To this end we consider the admissible proalgebraic subgroup $\{g \in G(F):
g\Lambda g^{-1}=\Lambda\}$, whose identity component we denote by
$K=K_{r,\lambda}$. Note that $K$ contains $A(\mathcal O)$. We write
$\mathfrak k=\mathfrak k_{r,\lambda}$ for the Lie algebra of $K$. By
construction the adjoint action of $K$ preserves $\Lambda$, so we may
consider the morphism 
\begin{equation}
\varphi=\varphi_{r,\lambda}:K\times \mathfrak a(F)_r \to \Lambda
\end{equation}
sending $(x,u)$ to $xux^{-1}$. We say that $\Lambda$ is a
\emph{root valuation lattice} if the morphism $\varphi$ is a
submersion. When $\Lambda$ is a root valuation lattice, the image
$\Lambda_0$ of $\varphi$ is an open Zariski dense subset of $\Lambda$,
and every element of $\Lambda_0$ is $K$-conjugate to an element of
$\mathfrak a(F)_r$.  

As a basic example, for the root valuation function  taking the value
$0$ on all roots, $\mathfrak g(\mathcal O)$ is a root valuation lattice. So
too is any parahoric subalgebra of $\mathfrak g(F)$ containing $\mathfrak
a(\mathcal O)$. The following proposition determines all root valuation
lattices. It makes use of the function $r_m:R \to \mathbb Z$ defined by 
\begin{equation}
r_m(\alpha):= \max\{r(\beta): \beta \text{ is not strongly orthogonal to
$\alpha$}\}. 
\end{equation} 
Note that 
\begin{equation}
r(\alpha) \le r_m(\alpha)
\end{equation}
since $\alpha$ is not strongly orthogonal to itself. 

\begin{proposition}\label{prop.rtvlat}
Let $r$ be a root valuation function and let $\lambda:R \to \mathbb Z$ be
any function. Define a third function $k:R \to \mathbb Z$ as the difference
$k:=\lambda-r$. Then $\Lambda_{r,\lambda}$ is a root valuation lattice if
and only if $k$ satisfies the following two conditions: 
\begin{enumerate}
\item $  k(\alpha)+k(-\alpha) \ge
r_m(\alpha)-r(\alpha)$  \quad for all $\alpha \in R$, 
\item $k(\alpha)+k(\beta)-k(\alpha+\beta) \ge
r(\alpha+\beta)-\min\{r(\alpha),r(\beta)\}$ \quad
for all $\alpha,\beta \in R$  such that $\alpha+\beta \in R$.
\end{enumerate} 
When $\Lambda_{r,\lambda}$ is a root valuation lattice, the Lie algebra of
$K_{r,\lambda}$ is given by  
\begin{equation}
\mathfrak k_{r,\lambda}=\mathfrak a(\mathcal O) \oplus
\bigl(\bigoplus_{\alpha \in R} P_\alpha^{k(\alpha)}\bigr).
\end{equation}
\end{proposition} 

\begin{proof} As usual we omit the subscripts $r,\lambda$ on
$\Lambda,K,\mathfrak k$. In reading this proof one needs to bear in mind
that $r(-\alpha)=r(\alpha)$ for all roots $\alpha$. Since
$K$ contains
$A(\mathcal O)$, the lattice
$\mathfrak k$ is $A(\mathcal O)$-stable, hence of the form 
\begin{equation}
\bigl(\mathfrak k \cap \mathfrak a(F)\bigr) \oplus \bigl( \bigoplus_{\alpha
\in R} P_\alpha^{l(\alpha)}\bigr)
\end{equation} 
for some function $l:R \to \mathbb Z$.  Moreover 
\begin{equation}
\mathfrak k=\{ X \in \mathfrak g(F): [X,\Lambda] \subset \Lambda\},
\end{equation}
from which it follows that 
\begin{equation}
\mathfrak k \cap \mathfrak a(F)=\mathfrak a(\mathcal O)
\end{equation}
and that $l(\alpha)$ is the smallest integer satisfying the conditions 
\begin{enumerate}
\item $l(\alpha)+r(\alpha) \ge \lambda(\alpha)$ 
\item $\epsilon^{l(\alpha)+\lambda(-\alpha)}H_\alpha \in \mathfrak
a(F)_{\ge r}$ 
\item $l(\alpha)+\lambda(\beta) \ge \lambda(\alpha+\beta)$ for all $\beta
\in R$ such that $\alpha+\beta \in R$. 
\end{enumerate} 
Here $H_\alpha$ denotes the coroot $\alpha^\vee$, viewed as an element
in
$\mathfrak a$. Expanding out the second condition, and using the equality
$\lambda=k+r$, we find that $l(\alpha)$ is the maximum of the (finite) set
$S=S_1\cup S_2 \cup S_3$ of integers, where 
\begin{enumerate}
\item $S_1=\{k(\alpha)\}$  
\item $S_2=\{r(\beta)-r(\alpha)-k(-\alpha):\beta \text{ is not
orthogonal to $\alpha$}\}$ 
\item $S_3=\{k(\alpha+\beta)+r(\alpha+\beta)-k(\beta)-r(\beta): \beta
\text{ satisfies $\alpha+\beta \in R$}\}$.
\end{enumerate} 

It follows from $K$-equivariance that $\varphi$ is a submersion if and only
if this is so at all points of the form $(1,u)$ with $u \in \mathfrak
a(F)_r$. Since $\varphi$ is a submersion at $(1,u)$ if and only if 
\begin{equation*}
\mathfrak a(F)_{\ge r}+[\mathfrak k,u]=\Lambda,
\end{equation*}
we conclude that $\Lambda$ is a root valuation lattice if and only if the
functions $l$ and $k$ coincide. In particular the last statement of the
proposition is now clear. 
Furthermore, in view of our description of
$l(\alpha)$ as a maximum of a certain set of integers, one of which is
$k(\alpha)$, we see that $\Lambda$ is a root valuation lattice if and only
if the following two conditions hold: 
\begin{enumerate} 
\item[(i)] 
$k(\alpha)+k(-\alpha)\ge r(\beta)-r(\alpha)$ whenever $\alpha$,$\beta \in
R$ are not orthogonal,
\item[(ii)] $k(\alpha) +k(\beta)-k(\alpha+\beta)\ge
r(\alpha+\beta)-r(\beta)$ for all $\alpha,\beta \in R$ such that $\alpha
+\beta \in R$. 
\end{enumerate}

In order to relate (i),(ii) to the somewhat different looking conditions
(1),(2) in the statement of the proposition, we are going to prove the
following claim: conditions (i) and (ii) together imply that 
\begin{equation}
k(\alpha)+k(-\alpha)\ge r(\gamma)-r(\alpha)
\end{equation}
whenever $\alpha,\gamma \in R$ are not strongly orthogonal, a statement
which is obviously equivalent to condition (1) of the proposition.  

For this we
apply (ii) to the roots $\alpha+\beta,-\alpha$ (whose sum is the root
$\beta$), obtaining the inequality 
\begin{equation}
k(\alpha+\beta)+k(-\alpha)-k(\beta)\ge r(\beta)-r(\alpha),
\end{equation}
 which, when added to (ii), yields the inequality 
\begin{equation}
k(\alpha)+k(-\alpha)\ge r(\alpha+\beta)-r(\alpha),
\end{equation}
valid whenever $\alpha+\beta$ is a root. In other words 
\begin{equation}
k(\alpha)+k(-\alpha)\ge r(\gamma)-r(\alpha)
\end{equation} 
whenever $\gamma$ is a root such that the difference of $\alpha$ and
$\gamma$ is a root. This last inequality also holds when the sum of
$\alpha$ and $\gamma$ is a root, as we see from the fact that the left side
of the inequality is invariant under replacing $\alpha$ by its negative,
 as is the term $r(\alpha)$ on the right side. We conclude that
\begin{equation}\label{eq.spam}
k(\alpha)+k(-\alpha)\ge r(\gamma)-r(\alpha)
\end{equation}
for any $\gamma \in R$ such that either $\alpha-\gamma$ or $\alpha+\gamma$
is root. In addition, taking $\alpha=\beta$ in (i), we see  that 
\begin{equation}\label{eq.morespam}
k(\alpha)+k(-\alpha)\ge 0.
\end{equation} 
Combining  \eqref{eq.spam} and \eqref{eq.morespam}, we conclude that 
\begin{equation}
k(\alpha)+k(-\alpha)\ge r(\gamma)-r(\alpha)
\end{equation}
whenever either $\alpha-\gamma$ or $\alpha+\gamma$ lies in $R \cup \{0\}$,
or, in other words, whenever $\alpha$ and $\gamma$ are not strongly
orthogonal. 

We have just shown that (i) and (ii)
imply condition (1) of the proposition. On the other hand, condition (1) of
the proposition trivially implies (i). 
At this point it remains only to observe that  condition (2) of the
proposition is equivalent to the conjunction of (ii) and the condition
obtained from (ii) by switching $\alpha$ and $\beta$.
\end{proof} 

\begin{remark}
As we just saw at the end of the proof the previous
proposition, we may work with (2) in the less symmetrical form 
\begin{equation}
k(\alpha)+k(\beta)-k(\alpha+\beta) \ge r(\alpha+\beta)-r(\alpha)
\end{equation} 
whenever it is convenient to do so. 
\end{remark}

\begin{remark}\label{rem.rtl}
If $\Lambda$ is a root valuation lattice for $r$, then $\epsilon^n\Lambda$ is a
root valuation lattice for the root valuation function $r+n$. Thus it is harmless
to work with root valuation functions taking values in $\mathbb Z_{\ge 0}$,
whenever it is convenient to do so. Now assume that $r$ does take
values in $\mathbb Z_{\ge 0}$. Then the lattice $\Lambda$ is contained in
$\mathfrak k$  (as is clear from Proposition \ref{prop.rtvlat}), and therefore
$\Lambda$ is a normal subalgebra of $\mathfrak k$ (since $[\mathfrak
k,\Lambda]\subset
\Lambda$).  Corollary \ref{cor.bt} shows that $\mathfrak k$ is
contained in some parahoric subalgebra and hence the same is  true of
$\Lambda$.  Moreover the discussion in subsection
\ref{sub.appl} shows that there is a connected admissible proalgebraic subgroup
$\tilde
\Lambda$ of $G(F)$ having Lie algebra $\Lambda$, and that
 a maximal torus in 
$\tilde
\Lambda$ is obtained by taking $S(\mathbb C)$, where $S$ is the subtorus of $A$
whose Lie algebra consists of all  $ u \in \mathfrak a(\mathbb C)$ such that 
$\alpha(u)=0$ for all 
$\alpha \in R_1=\{\alpha \in R: r(\alpha) \ge 1\}$ (in other words, $S$ is the
connected center of the Levi subgroup of $G$ containing $A$ and having $R_1$ as
its root system). These observations will  soon be used  in the proof of the key
lemma needed in the Conjugation Theorem for root valuation lattices. 
\end{remark}

\section{Conjugation Theorem}\label{sec.conjthm} 
\subsection{Setup for Conjugation Theorem}
Again we assume that $G$ is semisimple, and again we fix a root valuation function
$r:R \to \mathbb Z$. To $r$ are associated  subsets $R_n:=\{\alpha \in R:
r(\alpha) \ge n\}$ and  Levi subgroups $M_n$ containing $A$ such that the root
system of $M_n$ is
$R_n$. Write $L_n$ for the derived group of $M_n$ and $A_n$ for the connected
center of $M_n$. The corresponding Lie algebras are then related by $\mathfrak
m_n=\mathfrak l_n \oplus \mathfrak a_n$.  

We will consider 
some root valuation lattice $\Lambda$ for $r$. Recall that $\Lambda$ has the form
\[
\Lambda=\Lambda_A \oplus \bigl( \bigoplus_{\alpha \in R}
P_\alpha^{\lambda(\alpha)}\bigr),
\]
where $\Lambda_A=\mathfrak a(F)_{\ge r}=\{\sum_iu_i\epsilon^i: u_i \in \mathfrak
a_{i+1}(\mathbb C) \quad \forall i \in \mathbb Z \}$. 

As before we consider the normalizer $\mathfrak k$ of $\Lambda$ in $\mathfrak
g(F)$. Thus $\mathfrak k=\{ u \in \mathfrak g(F) : [u,\Lambda] \subset \Lambda
\}$. Again we denote by $K$ the connected admissible proalgebraic subgroup  whose
Lie algebra is $\mathfrak k$.

The goal of this section is to prove the following Conjugation Theorem. 
\begin{theorem}\label{thm.conj}
Let $r':R \to \mathbb Z$ be another root valuation function and assume that 
\[
|\{ \alpha \in R:r'(\alpha) \ge n\}| \le |R_n|
\] 
for all $n \in \mathbb Z$. Assume further that $u \in \Lambda$ lies in the
root valuation stratum $\mathfrak g(F)_{ r'}$. Then there exists $k \in K$ such
that
$k^{-1}uk \in \mathfrak a(F)_r$. In particular there exists an element $w$ in the
Weyl group $W$ such that $r'=wr$. 
\end{theorem}

\subsection{Key lemma}
The next lemma is the main step in the proof of the Conjugation Theorem. 
In the lemma we assume that $r$ takes nonnegative values, and therefore from
Remark \ref{rem.rtl}  we obtain a connected admissible proalgebraic subgroup
$\tilde\Lambda$ of $G(F)$ having Lie algebra $\Lambda$. 
\begin{lemma}\label{lem.keyconj}
Assume that our given root valuation function takes values in the set of
nonnegative integers. Let $u \in \Lambda$ and assume that $u$ is $G(F)$-conjugate
to an element $u'' \in \mathfrak a(F)$. Assume further that $u''$ is regular, so
that the root valuation function $r''(\alpha)=\val \alpha(u'')$ is defined. 
 Then
the following conclusions hold. 
\begin{enumerate}
\item $r''(\alpha) \ge 0$ for all $\alpha \in R$. 
\item $|\{\alpha \in R:r''(\alpha) \ge 1\}| \ge |R_1|$. 
\item If equality holds in the second item, then there exists $k \in \tilde\Lambda$
such that
$k^{-1}uk \in \mathfrak m_1(F)$. 
\end{enumerate}
\end{lemma} 
\begin{proof}
We know from Remark \ref{rem.rtl} that $\Lambda$ is contained in some parahoric
subalgebra, and therefore
$u$ and $u''$ are integral. It follows easily from the fourth part of Lemma 
\ref{lem.handy}  that
$u''
\in
\mathfrak a(\mathcal O)$, and the first item of the lemma follows.

We also know from Remark \ref{rem.rtl}  that $A_1(\mathbb
C)$ is a maximal torus in  $\tilde
\Lambda$. Therefore there exists $k \in \tilde\Lambda$  such that
$v:=k^{-1}u_sk$ lies in  
$\mathfrak a_1(\mathbb C)$, $u_s$ being the $\mathbb C$-semisimple part of $u$
(see subsection \ref{sub.tJd}). The second item of the lemma then follows from  
equation \eqref{eq.cardJ}. If equality holds in the second item, then the
set of roots in $R$ that vanish on $v$ must be precisely equal to $R_1$, the root
system of $M_1$. Thus the centralizer of $v$ in $G$ must in this case be $M_1$.
Since $k^{-1}uk$ centralizes $v$, we conclude that it lies in $\mathfrak m_1(F)$,
as desired. 
\end{proof}

\subsection{Proof of the conjugation theorem}
Now we prove the conjugation theorem. For each integer $n$ consider the
following statement. 
\begin{equation*}\label{eq.state}
\text{There exists $ k \in K $ such that $k^{-1} uk \in \mathfrak m_n(F)$}.
\tag{$S_n$} 
\end{equation*}
We claim that the statement \eqref{eq.state} is true for every integer $n$. This
is obvious when $n \ll 0$, so we may use induction on $n$. We now assume that
\eqref{eq.state} does hold for $n$ and will show that it also holds for $n+1$. 
It is harmless to replace $u$ by any conjugate under $K$, so we may as well assume
that $u$ itself lies in $\mathfrak m_n(F)$. Since $u$ is split regular semisimple
in $\mathfrak g(F)$, it is also split regular semisimple in $\mathfrak m_n(F)$. 

Denote by $r_n$ the restriction of $r$ to $R_n$; thus $r_n$ is a root valuation
function for the semisimple group $L_n$. Using Proposition \ref{prop.rtvlat}, one
checks that  $r_n$ and the integers
$\lambda(\alpha)$ ($\alpha\in R_n$) yield a root valuation lattice $\Lambda_n
\subset \mathfrak l_n(F)$ for $(L_n,r_n)$, and it is not difficult to see that 
\[
\Lambda \cap \mathfrak m_n(F) =\Lambda_n \oplus (\Lambda_A \cap \mathfrak a_n(F)).
\]
Since $u$ lies in $\Lambda \cap \mathfrak m_n(F)$, it decomposes uniquely as
$u=u_{L_n}+u_{A_n}$ with $u_{L_n} \in \Lambda_n$ and $u_{A_n} \in \Lambda_A \cap
\mathfrak a_n(F)$; clearly $u_{L_n}$ is split regular semisimple in $\mathfrak
l_n(F)$.  

Now $\Lambda'_n:=\epsilon^{-n}\Lambda_n$ is a root valuation lattice in $\mathfrak
l_n(F)$ for the root valuation function $r_n-n$, and since $r_n-n$ takes
nonnegative values, we may apply the key lemma to $L_n$, $r_n-n$ and the element
$\epsilon^{-n}u_{L_n}$, it being clear that $\epsilon^{-n}u_{L_n}$ is
$L_n(F)$-conjugate to an element $u'' \in (\mathfrak a\cap \mathfrak l_n)(F)$
that is regular for $L_n$ (since $u_{L_n}$ is split regular semisimple in
$\mathfrak l_n(F)$).

Denote by $r'':R_n \to \mathbb Z$ the root valuation function for $u''$. Since $u$
is conjugate under $L_n(F)$ to $u_{A_n}+\epsilon^nu''$ and is also
$G(F)$-conjugate to an element in $\mathfrak a(F)_{r'}$, we see that there exists
$w \in W$ such that the root valuation function for $u_{A_n}+\epsilon^nu''$ is
equal to $w(r')$. For $\alpha \in R_{L_n}=R_n$ we have $\val
\alpha(u_{A_n}+\epsilon^nu'')=n+\val \alpha(u'')$, which shows that
$r''(\alpha)+n=r'(w^{-1}\alpha)$. It follows that $r'(w^{-1}\alpha) \ge n+1$ for
any
$\alpha \in R_n$ such that $r''(\alpha) \ge 1$. Therefore 
\[
|\{\alpha \in R_n :r''(\alpha) \ge 1 \}| \le |\{ \alpha \in R: r'(w^{-1}\alpha)
\ge n+1 \}| \le |R_{n+1}|,
\]
showing that equality holds in the second item of the key lemma. By the third item
of that lemma we conclude that there exists
$k\in\tilde\Lambda'_n \subset L_n(F)$
 such that $k^{-1}(\epsilon^{-n}u_{L_n}) k \in \mathfrak
m_{n+1}(F)$, and it is then immediate that $k^{-1}uk \in \mathfrak m_{n+1}(F)$.
Moreover $k$ necessarily lies in $K$, since it follows easily from Proposition 
\ref{prop.rtvlat} that $\Lambda'_n \subset \mathfrak k$. 

Thus we have shown that the statement \eqref{eq.state} is true for every integer
$n$. Taking $n \gg 0$, we find that there exists $k \in K$ such that $k^{-1}uk \in
\mathfrak a(F)$. Since $K$ normalizes $\Lambda$, we also have $k^{-1}uk \in
\Lambda$, so that $k^{-1}uk$ actually lies in the intersection of $\Lambda$ and
$\mathfrak a(F)$, namely $\mathfrak a(F)_{\ge r}$. Let $\tilde r$ be the root
valuation function for
$k^{-1}uk$. Then  $\tilde r(\alpha) \ge
r(\alpha)$ for all $\alpha \in R$. Also by one of the hypotheses in the theorem
there exists $w \in W$ such that $r'=w\tilde r$, and we then have $r'(\alpha) \ge
(wr)(\alpha)$ for all $\alpha \in R$. By another hypothesis in the theorem we have 
\[
|\{\alpha \in R:r'(\alpha) \ge n \}| \le |R_n| =|wR_n|=|\{\alpha \in R:
(wr)(\alpha) \ge n \}|.
\]
From this, together with the fact that $r'(\alpha) \ge
(wr)(\alpha)$ for all $\alpha \in R$, it follows easily that $r'
=wr$, and the proof of the theorem is complete. \qed

\section{Existence of big root valuation lattices}\label{sec.big}
We are going to show that there exist root valuation lattices that are big in
a suitable sense. For this we need a number of preliminary results. 

\subsection{Operation $\perp$ on subsets of $R$}
 For any subset
$S$ of
$R$ we put 
\begin{equation}
S^\perp:=\{\beta \in R: \beta \text{ is strongly orthogonal to all roots
in $S$}\}.
\end{equation}
 
\begin{lemma}\label{lem.perp}
Let $S,T$ be  subsets of $R$. Then 
\begin{enumerate}
\item If $S \subset T$, then $S^\perp \supset T^\perp$. 
\item $S \subset S^{\perp\perp}$. 
\item $S^{\perp\perp\perp}=S^\perp$. 
\end{enumerate}
\end{lemma}
\begin{proof}
(1) and (2) are obvious, and together they imply (3). 
\end{proof}

\begin{lemma}
Let $\alpha \in R$. 
\begin{enumerate}
\item The centralizer of $\mathfrak g_\alpha$ in
$\mathfrak g$ is 
\begin{equation}\label{eq.acent}
\{u \in \mathfrak a: \alpha(u)=0\} \oplus \bigl(\bigoplus_{\beta \in
R:\alpha+\beta \notin R\cup\{0\}} \mathfrak g_\beta \bigr).
\end{equation} 
\item 
The centralizer in $\mathfrak g$ of the copy of $\mathfrak{sl}_2$ spanned
by $H_\alpha$, $\mathfrak g_\alpha$ and $\mathfrak g_{-\alpha}$ is 
\begin{equation}
\{u \in \mathfrak a: \alpha(u)=0\} \oplus \bigl(\bigoplus_{\beta \in
\{\alpha\}^\perp} \mathfrak g_\beta \bigr).
\end{equation}
\end{enumerate}
\end{lemma} 
\begin{proof}
Since $A$ normalizes $\mathfrak g_\alpha$, it normalizes the centralizer of
$\mathfrak g_\alpha$. Therefore this centralizer is the direct sum of its
intersections with the summands in the root space decomposition of
$\mathfrak g$. The decomposition \eqref{eq.acent} then follows from 
\begin{equation}[\mathfrak g_\alpha,\mathfrak g_\beta]=
\begin{cases}
\mathfrak g_{\alpha+\beta}&\text{if $\alpha+\beta \in R$,}\\
\mathbb C H_\alpha &\text{if $\alpha+\beta =0$,}\\
0 &\text{otherwise.}
\end{cases}
\end{equation} 
This proves (1), of which (2) is an immediate consequence. 
\end{proof}

\begin{corollary}\label{cor.pcl}
For any subset $S$ in $R$ the subset 
 $S^\perp$ is $\mathbb Z$-closed. 
\end{corollary}
\begin{proof}
Since intersections of $\mathbb Z$-closed subsets are $\mathbb Z$-closed,
it is enough to prove that $S^\perp$ is $\mathbb Z$-closed in the special
case $S=\{\alpha\}$, and this is clear from the second part of the previous
lemma. 
\end{proof}

\subsection{Non-archimedean functions $r:R \to \mathbb Z$} 
We say that a function $r:R \to \mathbb Z$ is \emph{non-archimedean} if for
every  $n \in \mathbb Z$ the set $R_n=\{\alpha \in R: r(\alpha) \ge n\}$ is
$\mathbb Z$-closed.  Clearly any root valuation function is non-archimedean. The
converse is true when
$R$ is of type
$A_n$, but not in general. 

It is evident that $r$ is a non-archimedean function if and only if 
\begin{equation}
r(-\alpha)=r(\alpha)
\end{equation}
 and 
\begin{equation}\label{eq.na}
r(\alpha+\beta) \ge
\min\{r(\alpha),r(\beta)\} \text{ whenever $\alpha+\beta  \in R$}.  
\end{equation}

\begin{remark}
For any non-archimedean function $r$, we have equality in \eqref{eq.na}
whenever $r(\alpha)\ne r(\beta)$. Indeed, we may as well assume that
$r(\alpha) >r(\beta)$, and then equality in \eqref{eq.na}  can be proved by
using 
\eqref{eq.na} twice, once for $\alpha,\beta$, and once for
$\alpha+\beta,-\alpha$ (whose sum is also a root). 
\end{remark}

Given any  function $r:R \to \mathbb Z$, we define another  
function
$r_m:R \to \mathbb Z$ by 
\begin{equation}\label{eq.r_m}
r_m(\alpha):= \max\{r(\beta): \beta \text{ is not strongly orthogonal to
$\alpha$}\}. 
\end{equation} 
Note that 
\begin{equation}
r(\alpha) \le r_m(\alpha)
\end{equation}
since $\alpha$ is not strongly orthogonal to itself. 

Now put $r'=-r_m$. Once again we put $R_n=\{\alpha \in R:r(\alpha)\ge n\}$,
and, in the same way, we put $R'_n:=\{\alpha \in R:r'(\alpha) \ge n \}$.

\begin{lemma}
The set $R'_n$ is equal to $(R_{1-n})^\perp$. Moreover $r'$ is a 
non-archimedean function. 
\end{lemma}
\begin{proof} 
To prove the first statement, just unwind the definitions. The second
statement follows from the first, together with Corollary \ref{cor.pcl}.  
\end{proof}

\begin{remark}
It is not always the case that $r'$ is a root valuation function, even
when $r$ is itself a root valuation function. This is clear from the lemma
we just proved, since there exist $\mathbb Q$-closed subsets $S$ such
that
$S^\perp$ is not $\mathbb Q$-closed.  
\end{remark} 

\begin{remark}
The notion of root valuation function is more useful
than that of  non-archimedean function. The notion of non-archimedean
function has been introduced only in order to have a convenient way of
referring to the properties of $r_m$ that are encoded in the fact that $r'$
is a non-archimedean function. 
\end{remark}

Later we will make use of the following result. 

\begin{lemma}\label{lem.k0prep}
 Let $r:R \to \mathbb Z$ be a non-archimedean function. 
Let $\alpha$,$\beta$,$\gamma$ be three roots, one of which is equal to the
the sum of the other two. Then 
\begin{enumerate}
\item $\alpha$,$\beta$,$\gamma$ are pairwise non-strongly orthogonal,
\item $r(\alpha)+r(\beta)-r(\gamma) \le
\max\{r(\alpha),r(\beta),r(\gamma)\}$, 
\item $\max\{r(\alpha),r(\beta),r(\gamma)\}\le
\min\{r_m(\alpha),r_m(\beta),r_m(\gamma)\}$,
\item $\min\{r_m(\alpha),r_m(\beta),r_m(\gamma)\}\le
r_m(\alpha)+r_m(\beta)-r_m(\gamma)$. 
\end{enumerate} 
Here $r_m$ is obtained from $r$ as above in \eqref{eq.r_m}. 
\end{lemma}
\begin{proof}
(1) is clear, since each of the three roots is either a sum or difference
of the other two, and (3) follows immediately from (1). 

Since $r$ is non-archimedean, the triple of integers
$r(\alpha)$,$r(\beta)$,$r(\gamma)$ is quite special: either all three are
equal, or two of them are equal and the other one is strictly larger. From
this (2) follows at once. Finally, (4) follows from (2), applied to the
non-archimedean function
$r'$. 
\end{proof}

\subsection{Some big root valuation lattices}
We have  seen that in order to get a root valuation lattice
$\Lambda_{r,\lambda}$ (for a given root valuation function $r$), we need a
function
$k:R\to
\mathbb Z$ satisfying the two conditions 
\begin{enumerate}
\item $  k(\alpha)+k(-\alpha) \ge
r_m(\alpha)-r(\alpha)$  \quad for all $\alpha \in R$, 
\item $k(\alpha)+k(\beta)-k(\alpha+\beta) \ge
r(\alpha+\beta)-\min\{r(\alpha),r(\beta)\}$ \quad
for all $\alpha,\beta \in R$  such that $\alpha+\beta \in R$.
\end{enumerate}

It is then natural to try to make the quantities on the lefthand sides of
the inequalities in (1) and (2) as small as possible. A first thought would
be to try to find $k$ for which equality holds in both conditions, but it is
easy to see that this is usually impossible.

So we need to try something else. In order to make $\Lambda_{r,\lambda}$ big, we
try to enforce equality in (1) without worrying about (2). Define a function $k_0$ 
on $R$ as follows:  
\begin{equation}
k_0(\alpha):=(r_m(\alpha)-r(\alpha))/2. 
\end{equation} 
For $k_0$ it is clear that equality does hold in (1), but 
of course $k_0$ takes values in $\frac{1}{2}\mathbb Z$, rather than $\mathbb Z$,
as we would have liked. We fix this  by putting 
\begin{equation}
k_1:=\lceil k_0 \rceil.
\end{equation}
\begin{lemma}
Both  $k_0$ and $k_1$ satisfy conditions \textup{(1)} and \textup{(2)}.
Since $k_1$ is integer-valued, it then provides us with a root valuation
lattice $\Lambda_{r,\lambda}$, where $\lambda=r+k_1$.  
\end{lemma}
\begin{proof}
It is clear that (1) holds for both functions. Next we verify (2). Consider
roots $\alpha$,$\beta$ such that 
 $\alpha+\beta$ is also a root, call it $\gamma$. We must show that 
\begin{equation}\label{eq.check}
k(\alpha)+k(\beta)-k(\gamma) \ge r(\gamma)-r(\alpha)
\end{equation}
for $k=k_0$ and $k=k_1$. We start with $k_0$. Using its explicit
definition, we see  that the inequality we must check is 
\begin{equation}\label{eq.long}
r_m(\alpha)+r_m(\beta)-r_m(\gamma) \ge r(\beta)+r(\gamma)-r(\alpha), 
\end{equation}
and
this follows from Lemma \ref{lem.k0prep}, once we note that the hypothesis
of that lemma is symmetrical in the three roots, so that the inequalities
stated in the conclusion of the lemma remain valid when the three roots are
permuted. 

This takes care of $k_0$. What about $k_1$? Obviously, for any root
$\delta$, the integer   
$k_1(\delta)$ is either $k_0(\delta)$ or $k_0(\delta)+1/2$, according as
$k_0(\delta)$ is integral or half-integral. Therefore, when we pass from
$k_0$ to $k_1$, if the lefthand side of \eqref{eq.check} decreases at all,
it can only decrease by 
$1/2$, and if this is the case, the lefthand side started out by being 
half-integral, and therefore even after decreasing by $1/2$ it remains
bigger than or equal to  the righthand side (simply because the
righthand side is an integer).  Therefore
\eqref{eq.check} also holds for
$k_1$. 
\end{proof}

\bibliographystyle{amsalpha}
\providecommand{\bysame}{\leavevmode\hbox to3em{\hrulefill}\thinspace}

\end{document}